\documentclass{article}

\usepackage{amssymb,amsmath,eufrak,amsthm,amscd}
\usepackage[matrix,arrow,curve]{xy}

\newcommand{\dcspec}{\operatorname{Spec^\Sigma}}
\newcommand{\Spec}{\operatorname{Spec}}
\newcommand{\PSpec}{\operatorname{PSpec}}
\newcommand{\Rad}{\operatorname{Rad}}
\newcommand{\PId}{\operatorname{Id^\Sigma}}
\newcommand{\PRad}{\operatorname{Rad^\Sigma}}
\newcommand{\Max}{\operatorname{Max}}
\newcommand{\Id}{\operatorname{Id}}
\newcommand{\PMax}{\operatorname{PMax}}
\newcommand{\Hom}{\operatorname{Hom}}

\newcommand{\Fun}{\operatorname{F}}

\newtheorem{theorem}{Theorem}
\newtheorem{lemma}[theorem]{Lemma}
\newtheorem{proposition}[theorem]{Proposition}
\newtheorem{corollary}[theorem]{Corollary}

\newtheorem*{statement*}{Statement}
\newtheorem*{theorem*}{Theorem}
\newtheorem*{lemma*}{Lemma}
\newtheorem*{fact*}{Fact}

\theoremstyle{definition}

\newtheorem*{definition*}{Definition}

\theoremstyle{remark}
\newtheorem{remark}[theorem]{Remark}
\newtheorem*{remark*}{Remark}
\newtheorem{example}[theorem]{Example}

\newtheorem*{example*}{Example}

\title{Difference Nullstellensatz in the case of finite group\footnote{\Xy-pic package is used}}
\date{}
\author{Dmitry Trushin\footnote{This work was partially supported by the NSF grants CCF-0964875 and CCF-0952591}}

\begin{document}

\maketitle

\begin{abstract}
We develop a geometric theory for difference equations with a given
group of automorphisms. To solve this problem we extend the class of
difference fields to the class of absolutely flat simple difference
rings called pseudofields. We prove the Nullstellensatz over
pseudofields and investigate geometric properties of
pseudovarieties.
\end{abstract}

\section{Introduction}

Our purpose is to produce a geometric technique allowing to obtain a
Picard-Vessiot theory of difference equations with difference
parameters. Unfortunately, usual geometric approaches do not allow
to produce a Picard-Vessiot extension with difference parameters.
Roughly speaking, the problem is that morphisms of varieties in
these theories are not constructible. Therefore, we have to develop
an absolutely new machinery. The main advantage of our theory is
that the morphisms of varieties are constructible. We also describe
general properties of our varieties. Using these results one can
obtain a Picard-Vessiot theory of difference equations with
difference parameters~\cite{AOT}. The most important application of
this theory is the description of difference relations among
solutions of difference equations, especially, for Jacobi's
theta-function.

This article is devoted to producing a general geometric theory of
difference equations. Therefore, it is hard to distinguish main
results. Nevertheless, we underline the following ones.
Theorem~\ref{equtheor} describes difference closed pseudofields.
Proposition~\ref{homst} used to obtain all geometric results,
particulary, this proposition shows that morphisms of
pseudovarieties are constructible. Describing the global regular
functions, Theorem~\ref{regularf} reduces the geometric theory to
the algebraic one. Lemma~\ref{taylormod} appears in different forms
in the text (for example, Proposition~\ref{taylor}), this full
version allows to connect the theory of pseudovarieties with the
theory of algebraic varieties. A more detailed survey of the main
results is included in the following section.

In this paper, we develop a geometric theory of difference
equations. But what does it mean? In commutative algebra if an
algebraically closed field is given, there is a one-to-one
correspondence between the set of affine algebraic varieties and
radical ideals in a polynomial ring. Moreover, the category of
affine algebraic varieties is equivalent to the category of finitely
generated algebras. Such a point of view was extended to the case of
differential equations by Ritt and his successors. The notion of
differential algebraic variety led to the notion of differentially
closed field. The similar results appeared in difference algebra.
The first results in this direction were obtained by
Cohn~\cite{Cohn}. He discovered the following difficulty: to obtain
a necessarily number of solutions of a given system of difference
equations we have to consider solutions in several distinct
difference fields at the same time. This effect prevented from
finding the notion of difference closed field. However, such notion
was introduced by Macintyer~\cite{Maci}. In model theory the theory
of a difference closed field is called ACFA. A detailed survey of
ACFA theory can be found in~\cite{Chad}. The appearance of the
notion of difference closed field allowed to define the notion of
difference variety in the same manner as in differential case.

In~\cite{Hrush} Hrushovski develops the notion of difference
algebraic geometry in a scheme theoretic language. But his machinery
can be applied only to well-mixed rings. Unfortunately, here is a
deeper problem: all mentioned attempts of building a geometric
theory deal with fields. Let us recall the main difficulties: 1)
there exist a pair of difference fields such that there is no
difference overfield containing both of them 2) morphisms of
difference varieties are not constructible 3) a maximal difference
ideal is not necessarily prime 4) the set of all common zeros of a
nontrivial difference ideal is sometimes empty. An example can be
found in~\cite[Proposition~7.4]{Rosen}. The essential idea is to
extend the class of difference fields. Such approach was used in
~\cite{vPS,Tace,AmMas,Wib}. In particulary, in the Picard-Vessiot
theory of difference equations one founds that the ring containing
enough solutions is not necessarily a field but rather a finite
product of fields. Solving a similar problem, Takeuchi considered
simple Artinian rings. In~\cite{Wib} Wibmer combined the ideas of
Hrushovski with those similar to the ideas in Takeuchi's work to
develop the Galois theory of difference equations in the most
general context. All these works ultimately deal with finite
products of fields.

As we can see, the Galois theory of difference equations requires to
consider finite products of fields. This simple improvement allowed
to produce difference-differential Picard-Vessiot theory with
differential parameters~\cite{HardSin}. The next step is to obtain
Picard-Vessiot theory for the equations with difference parameters.
A first idea of how to do this is to use difference closed fields
and to repeat the discourse developed in~\cite{vPS,HardSin}.
Unfortunately, this method does not work. And the general problem is
that the morphisms of difference varieties are not constructible.
This effect appears when we construct a Picard-Vessiot extension for
an equation with difference parameters. In this situation we expect
that the constants of the extension coincide with the constants of
the initial field. And we use this fact to produce a Galois group as
a linear algebraic group. However, the constants of a Picard-Vessiot
extension need not be a field, for example~\cite[example~2.6]{AOT}.

Therefore, we must develop a new geometric theory. The first
question is what class of difference rings is appropriate for our
purpose. The answer is the class of all simple absolutely flat
difference rings. The detailed discussion of how to figure this out
can be found in~\cite{DNull}. We shall use a term pseudofield for
such rings. Our plan is to introduce difference closed pseudofields
and to develop the corresponding theory of pseudovarieties.

Here we shall briefly discuss some milestones of the theory. First
of all we need the notion of difference closed pseudofield. A
similar problem appears in differential algebra of prime
characteristic. In prime characteristic, we have to deal with
quasifields instead of fields. In~\cite{Quasi}, differentially
closed quasifields are introduced. The crucial role in this theory
is played by the ring of Hurwitz series. A difference algebraic
analogue is introduced in Section~\ref{sec33} and is called the ring
of functions. Such a construction appeared in many papers, for
example~\cite{vPS,Mor,UmMor}.

Here our theory is divided into two parts: the case of a finite
group or an infinite one. This paper deals with the finite groups.
We show that functions on the group give a full classification of
difference closed pseudofields. In this situation, pseudofields are
finite products of fields. Therefore, pseudofields in our sense and
pseudofields in~\cite{Wib} coincide. The case of infinite group is
much harder and is scrutinized in~\cite{DNull}.

The theory of difference rings has one important technical
difficulty: we cannot use an arbitrary localization. For example,
suppose that we need to investigate an inequality $f\neq0$. To do
this one can consider the localization with respect to the powers of
$f$. Unfortunately, the constructed ring is not necessarily a
difference one. To find the ``minimal'' difference ring containing
$1/f$, we should generate the smallest invariant multiplicative
subset by $f$. But this subset often contains zero. Therefore, we
have to develop a new machinery to avoid this difficulty. This
machinery is developed in section~\ref{sec42} and is called an
inheriting. Roughly speaking, all results of the paper are based on
the classification of difference closed pseudofields and the
inheriting machinery.

\subsection{Structure of the paper}

All necessary terms and notation are introduced in
Section~\ref{sec2}. Section~\ref{sec3} is devoted to the basic
techniques used in further sections. In Section~\ref{sec31}, we
introduce pseudoprime ideals and investigate their properties. In
the next Section~\ref{sec32}, we deal with pseudospectra and
introduce a topology on them. In Section~\ref{sec33}, the most
important class of difference rings is presented. We prove the
theorem of the Taylor homomorphism for this class of rings
(Proposition~\ref{taylor}).

The most interesting case for us is the case of finite groups of
automorphisms. Section~\ref{sec4}. In Section~\ref{sec41}, we
improve basic technical results obtained in Section~\ref{sec3}.
Section~\ref{sec42} provides the relation between the commutative
structure of a ring and its difference structure. Since in
difference algebra we are not able to produce fraction rings with
respect to an arbitrary multiplicatively closed sets, we need an
alternative technique, which is based on the inheriting of
properties. The main technical result is Proposition~\ref{inher}
allowing to avoid localization.

The structure of pseudofields is scrutinized in Section~\ref{sec43}.
We introduce difference closed pseudofields and classify them up to
isomorphism (Proposition~\ref{equtheor}). We prove that every
pseudofield (so, thus every field) can be embedded into a difference
closed pseudofield (Propositions~\ref{difclosemin}
and~\ref{difclosuni}). Our technique is illustrated by a sequence of
examples. Section~\ref{sec44} plays an auxiliary role. Its results
have special geometric interpretation. The most important statements
are Proposition~\ref{homst} and its corollaries~\ref{cor1}
and~\ref{cor2}.

Using difference closed pseudofield one can produce a geometric
theory of difference equations with finite group of automorphisms.
In Section~\ref{sec45}, we introduce the basic geometric notions.
The main result of the section is the difference Nullstellensatz for
pseudovarieties (Proposition~\ref{nullth}). In Section~\ref{sec46},
we construct two different structure sheaves of regular functions.
The first one consists of functions that are given by a fraction
$a/b$ in a neighborhood of each point. Every pseudofield has an
operation generalizing division. We use this operation to produce
the second sheaf. And the main result is that these sheaves coincide
and the ring of global sections consists of polynomial functions.
Section~\ref{sec47} contains nontrivial geometric results about
pseudovarieties. For example, Theorem~\ref{constrst} says that
morphisms are constructible.

There is a natural way to identify a pseudoaffine space with an
affine space over some algebraically closed field. Thus, every
pseudovariety can be considered as a subset of an affine space. One
can show that pseudovarieties are closed in the Zariski topology.
Moreover, there is a one-to-one correspondence between
pseudovarieties and algebraic varieties. We prove this in
Section~\ref{sec48} and we show how to derive geometric properties
of pseudovarieties using the adjoint variety in Section~\ref{sec49}.
The final section contains the basic results on the dimension.

\section{Terms and notation}\label{sec2}

This section is devoted to basic notions and objects used further.
We shall define an interesting for us class of rings and the notion
of pseudospectrum.

Let $\Sigma$ be an arbitrary group. A ring $A$ will be said to be a
difference ring if $A$ is an associative commutative ring with an
identity element such that the group $\Sigma$ is acting on $A$ by
means of ring automorphisms. A difference homomorphism of difference
rings is a homomorphism preserving the identity element and
commuting with the action of $\Sigma$. A difference ideal is an
ideal stable under the action of the group $\Sigma$. We shall write
$\Sigma$ instead of the word difference. A simple difference ring is
a ring with no nontrivial difference ideals. The set of all
$\Sigma$-ideals of $A$ will be denoted by $\PId A$. For every ideal
$\frak a\subseteq A$ and every element $\sigma\in\Sigma$ the image
of $\frak a$ under $\sigma$ will be denoted by $\frak a^\sigma$.

The set of all, radical, prime, maximal ideals of $A$ will be
denoted by $\Id A$, $\Rad A$, $\Spec A$, $\Max A$, respectively. The
set of all prime difference ideals of $A$ will be denoted by
$\dcspec A$. For every ideal $\frak a\subseteq A$ the largest
$\Sigma$-ideal laying in $\frak a$ will be denoted by $\frak
a_\Sigma$. Such an ideal exists because it coincides with the sum of
all difference ideals contained in $\frak a$. Note that
$$
\frak a_\Sigma=\{\, a\in\frak a\mid\forall \sigma\in\Sigma\colon
\sigma(a)\in\frak a \,\}.
$$
So, we have a mapping
$$
\pi\colon \Id A\to \PId A
$$
defined by the rule $\frak a\mapsto \frak a_\Sigma$. Straightforward
calculation shows that for every family of ideals $\frak a_\alpha$
we have
$$
\pi(\bigcap_{\alpha} \frak a_\alpha)=\bigcap_{\alpha} \pi(\frak
a_\alpha).
$$
It is easy to see that for any ideal $\frak a$ there is the equality
$$
\frak a_\Sigma=\bigcap_{\sigma\in\Sigma} \frak a^\sigma.
$$

We shall define the notion of pseudoprime ideal of a $\Sigma$-ring
$A$. Let $S\subseteq A$ be a multiplicatively closed subset
containing the identity element, and let $\frak q$ be a maximal
$\Sigma$-ideal not meeting $S$. Then the ideal $\frak q$ will be
called pseudoprime. The set of all pseudoprime ideals will be
denoted by $\PSpec A$ and is called a pseudospectrum.

Note that the restriction of $\pi$ onto the spectrum gives the
mapping
$$
\pi\colon \Spec A\to \PSpec A.
$$
The ideal $\frak p$ will be called $\Sigma$-associated with
pseudoprime $\frak q$ if $\pi(\frak p)=\frak q$. Let $\frak q$ be a
pseudoprime ideal, and let $S$ be a multiplicatively closed set from
the definition of $\frak q$, then every prime ideal containing
$\frak q$ and not meeting $S$ is $\Sigma$-associated with $\frak q$.
So, the mapping $\pi\colon \Spec A\to \PSpec A$ is surjective.

Let $S$ be a multiplicatively closed set and $\frak a$ be an ideal
of $A$. Then the saturation of $\frak a$ with respect to $S$ will be
the following ideal
$$
S(\frak a)=\bigcup_{s\in S}(\frak a:s).
$$
If $s$ is an element of $A$ then the saturation of $\frak a$ with
respect to $\{s^n\}$ will be denoted by $\frak a:s^\infty$.

If $S$ is a multiplicatively closed  subset of $A$ then the ring of
fractions of $A$ with respect to $S$ will be denoted by $S^{-1}A$.
If $S=\{t^n\}_{n=0}^\infty$ then the ring $S^{-1}A$ will be denoted
by $A_t$. If $\frak p$ is a prime ideal of $A$ and  $S=A\setminus
\frak p$ then the ring $S^{-1}A$ will be denoted by $A_{\frak p}$.

For any subset $X\subseteq A$ the smallest difference ideal
containing $X$ will be denoted by $[X]$. The smallest radical
difference ideal containing $X$ will be denoted by $\{X\}$. The
radical of an ideal $\frak a$ will be denoted by $\frak r(\frak a)$.
So, we have that $\{X\}=\frak r([X])$.

Let $f\colon A\to B$ be a homomorphism of rings and let $\frak a$
and $\frak b$ be ideals of $A$ and $B$, respectively. Then we define
the extension $\frak a^e$ to be the ideal $f(\frak a)B$ generated by
$f(\frak a)$. The contraction $\frak b^c$ is the ideal $f^*(\frak
b)=f^{-1}(\frak b)$. If the homomorphism $f\colon A\to B$ is a
difference one then both extension and contraction of difference
ideals are difference ones.

Let $f\colon A\to B$ be a $\Sigma$-homomorphism of difference rings,
and let $\frak q$ be a pseudoprime ideal of $B$. The contraction
$\frak q^c$ is pseudoprime because $\pi$ is surjective and commutes
with $f^*$. So, we have a mapping from $\PSpec B$ to $\PSpec A$.
This mapping will be denoted by $f^*_\Sigma$. It follows from the
definition that the following diagram is commutative
$$
\xymatrix{
    \Spec B\ar[r]^{f^*}\ar[d]^{\pi}&\Spec A\ar[d]^{\pi}\\
    \PSpec B\ar[r]^{f^*_\Sigma}&\PSpec A\\
}
$$

The set of all radical $\Sigma$-ideals of $A$ will be denoted by
$\PRad A$. For the convenience maximal difference ideals will be
called pseudomaximal. This set will be denoted by $\PMax A$. It is
clear, that every pseudomaximal ideal is pseudoprime ($S=\{1\}$). It
is easy to see, that a radical difference ideal can be presented as
an intersection of pseudoprime ideals. So, the objects with prefix
pseudo have the same behavior as the objects without it.

The ring of difference polynomials $A\{Y\}$ is a ring $A[\Sigma Y]$,
where $\Sigma$ acts in the natural way. A difference ring $B$ will
be called an $A$-algebra if there is a difference homomorphism $A\to
B$. It is clear, that every $A$-algebra can be presented as a
quotient ring of some polynomial ring $A\{Y\}$.

\section{Basic technique}\label{sec3}

In this section we shall prove basic results about the introduced
set of difference ideals.

\subsection{Pseudoprime ideals}\label{sec31}

\begin{proposition}\label{TechnicalStatement}
Let $\frak q$ and $\frak q'$ be pseudoprime ideals of a difference
ring $A$. Then
\begin{enumerate}
\item Ideal $\frak q$ is radical.
\item For every ideal $\frak p$ $\Sigma$-associated  with $\frak q$
there is the equality
$$
\frak q=\bigcap_{\sigma\in \Sigma}\frak p^\sigma.
$$
\item For every element $s\notin \frak q$ there is the equality
$$
(\frak q:s^\infty)_\Sigma=\frak q.
$$
\item  It follows from the equality $\frak q:s^\infty=\frak
q':s^\infty$ that for every element $s$ either $s$ belongs to $\frak
q$ and $\frak q'$, or $\frak q=\frak q'$.
\item For every two difference ideals $\frak a$ and $\frak b$ the
inclusion $\frak a\frak b\subseteq \frak q$ implies either $\frak
a\subseteq \frak q$, or $\frak b\subseteq \frak q$.
\end{enumerate}
\end{proposition}
\begin{proof}
(1). Let $S$ be a multiplicatively closed subset of $A$ such that
$\frak q$ is a maximal difference ideal not meeting $S$. Then $\frak
r(\frak q)$ is a difference ideal containing $\frak q$ and not
meeting $S$. Consequently, $\frak r(\frak q)=\frak q$.

(2). The equality $\frak p_\Sigma=\cap \frak p^\sigma$ is always
true. But from the definition we have $\frak q=\frak p_\Sigma$.

(3). Let $\frak p$ be a $\Sigma$-associated with $\frak q$ prime
ideal. Then it follows from~(2) that there exists $\sigma\in\Sigma$
such that $s\notin \frak p^\sigma$.  Therefore, there is the
inclusion
$$
(\frak q:s^\infty)\subseteq(\frak p^\sigma:s^\infty)=\frak p^\sigma,
$$
and, consequently,
$$
(\frak q:s^\infty)_\Sigma\subseteq\frak p^\sigma_\Sigma=\frak q.
$$
The other inclusion is obvious.

Note that for every ideal $\frak a$ the equality $\frak
a:s^\infty=A$  holds if and only if $s\in \frak a$. Therefore, we
need to consider the case $s\notin \frak q$ and $s\notin \frak q'$.
From the previous item we have
$$
\frak q=(\frak q:s^\infty)_\Sigma=(\frak q':s^\infty)_\Sigma=\frak
q'.
$$

(5). Let $\frak p$ be a $\Sigma$-associated with $\frak q$ prime
ideal. Then either $\frak a\subseteq \frak p$, or $\frak b\subseteq
\frak p$. Suppose that the first one holds. Then
$$
\frak a=\frak a_\Sigma\subseteq \frak p_\Sigma=\frak q.
$$
\end{proof}

We shall show that condition~(3) does not hold for an arbitrary
multiplicatively closed subset $S$.

\begin{example}
Let $\Sigma=\mathbb Z$, consider the ring $A=K^\Sigma$, where $K$ is
a field. Then this ring is of Krull dimension zero. So, every prime
ideal is maximal. This is a well-known fact that the maximal ideals
of $A$ can be described in terms of maximal filters on $\Sigma$.
Namely, for an arbitrary filter $\mathcal F$ of $\Sigma$  we define
the ideal
$$
\frak m_{\mathcal F}=\{\,x\in A\mid \{\,n\mid x_n=0\,\}\in\mathcal
F\,\}.
$$
There are two different types of maximal ideals. The first type
corresponds to principal maximal filters
$$
\frak m_k=\{\,x\in A\mid x_k=0\,\}
$$
and the second type corresponds to ultrafilters $\frak m_{\mathcal
F}$. It is clear that for all ideals of the first type we have
$(\frak m_k)_\Sigma=0$. But for any ultrafilter $\mathcal F$ the
ideal $\frak m_{\mathcal F}$  contains the ideal $K^{\oplus\Sigma}$
consisting of all finite sequences. Therefore,  $(\frak m_{\mathcal
F})_\Sigma\neq0$. As we can see not every minimal prime ideal
containing zero ideal is $\Sigma$-associated with it. Additionally,
set  $S=A\setminus \frak m_{\mathcal F}$, where $\mathcal F$ is an
ultrafilter. Then
$$
(S(0))_\Sigma=(\frak m_{\mathcal F})_\Sigma\neq0.
$$
\end{example}

Let us note one peculiarity of radical difference ideals.

\begin{example}
Let $\Sigma = \mathbb Z$. Consider the ring $A=K\times K$, where
$\Sigma$ acts as a permutation of factors. Then
$$
\{(1,0)\}\{(0,1)\}\nsubseteq \{(1,0)(0,1)\},
$$
because the left-hand part is $A$ and the right-hand part is $0$.
So, the condition $\{X\}\{Y\}\subseteq\{XY\}$ does not hold.
\end{example}

\subsection{Pseudospectrum}\label{sec32}

We shall provide a pseudospectrum with a structure of topological
space such that the mapping $\pi$ will be continuous.

Let $A$ be an arbitrary difference ring and $X$ be the set of all
its pseudoprime ideals. For every subset $E\subseteq A$ let $V(E)$
denote the set of all pseudoprime ideals containing $E$.

\begin{proposition}\label{pspectop}
Using the above notation the following holds:
\begin{enumerate}
\item If $\frak a$ is a difference ideal generated by $E$, then
$$
V(E)=V(\frak a)=V(\frak r(\frak a)).
$$
\item $V(0)=X$, $V(1)=\emptyset$.
\item Let $(E_i)_{i\in I}$  be a family of subsets of $A$.
Then
$$
V\left(\bigcup_{i\in I}E_i\right)=\bigcap_{i\in I}V\left(E_i\right).
$$
\item For any difference ideals $\frak a$, $\frak b$ in $A$ the
following holds
$$
V(\frak a\cap\frak b)=V(\frak a\frak b)=V(\frak a)\cup V(\frak b).
$$
\end{enumerate}
\end{proposition}
\begin{proof}
Condition~(1) immediately follows from the definition of $V(E)$ and
the fact that pseudoprime ideal is radical. Conditions~(2) and~(3)
are obvious. The last statement immediately follows from
condition~(5) of Proposition~\ref{TechnicalStatement}.
\end{proof}

So, we see that the sets $V(E)$ satisfy the axioms for closed sets
in topological space. We shall fix this topology on pseudospectrum.
Consider the mapping
$$
\pi\colon \Spec A\to \PSpec A.
$$
For every difference ideal $\frak a$ we have
$$
\pi^{-1}(V(\frak a))=V(\frak a),
$$
i.~e., the mapping $\pi$ is continuous. Let us recall that $\pi$ is
always surjective.

Let us denote the pseudospectrum of a difference ring $A$ by $X$.
Then for every element $s\in A$ the complement of $V(s)$ will be
denoted by $X_s$. From the definition of topology we have that every
open subset can be presented as a union of the sets of the form
$X_s$. In other words the family $\{X_s\mid s\in A\}$ forms a basis
of topology. It should be noted that the intersection $X_s\cap X_t$
is not necessarily of the form $X_u$.

\begin{proposition}\label{pspecst}
Using the above notation we have
\begin{enumerate}
\item $X_s\cap X_t=\cup_{\sigma,\tau\in\Sigma} X_{\sigma(s) \tau(t)}$.
\item $X_s=\emptyset$ iff  $s$ is nilpotent.
\item $X$ is quasi-compact (that is, every open covering of $X$ has
a finite subcovering).
\item There is a one-to-one correspondence between the set of all closed
subsets of the pseudospectrum and the set of all radical difference
ideals:
$$
\frak t\mapsto V(\frak t)\:\mbox{ è }\:V(E)\mapsto \bigcap_{\frak
q\in V(E)} \frak q.
$$
\end{enumerate}
\end{proposition}
\begin{proof}
Condition~(1) is proved by straightforward calculation.

(2). Note that $X_s$ is not empty if and only if the set of all
prime ideals not containing $s$ is not empty. The last condition is
equivalent to $s$ being not nilpotent.

(3). Let  $\{V(\frak a_i)\}$ be a centered family of closed subsets
(that is every intersection of finitely many elements is not empty),
where $\frak a_i$ are difference ideals. We need to show that
$\cap_i V(\frak a_i)$  is not empty. Suppose that contrary holds
$\cap_i V(\frak a_i)=\emptyset$. But
$$
\bigcap_i V(\frak a_i)=V(\sum_i \frak a_i)=\emptyset.
$$
The last equality is equivalent to condition that $1$ belongs to
$\sum_i\frak a_i$. But in this situation $1$ belongs to a finite
sum. Therefore, the corresponding intersection of finitely many
closed subsets is empty, contradiction.

(4). The statement immediately follows from the equality
$$
\frak r([E])=\bigcap_{\frak q\in V(E)} \frak q.
$$
Let us show that this equality holds. The inclusion $\subseteq$ is
obvious. Let us show the other one. Let $g$ not belong to the
radical of $[E]$ then consider the set of all difference ideals
containing $E$ and not meeting $\{g^n\}_{n=0}^\infty$. This set is
not empty, since $[E]$ is in it. From Zorn's lemma there is a
maximal difference ideal with that property. From the definition
this ideal is pseudoprime.
\end{proof}

\subsection{Functions on the group}\label{sec33}

For every commutative ring $B$ the set of all functions from
$\Sigma$ to $B$ will be denoted by $\Fun B$. As a commutative ring
it coincides with the product $\prod_{\sigma\in\Sigma} B$. Let us
provide $\Fun B$ with the structure of a difference ring. We define
$\sigma(f)(\tau)=f(\sigma^{-1}\tau)$. For every element $\sigma$ of
the group $\Sigma$ there is a homomorphism
\begin{equation*}
\begin{split}
\gamma_{\sigma}\colon\Fun B&\to B\\
f&\mapsto f(\sigma)
\end{split}
\end{equation*}
It is clear that $\gamma_{\tau}(\sigma
f)=\gamma_{\sigma^{-1}\tau}(f)$.

\begin{proposition}\label{taylor}
Let $A$ be a difference ring, and let $\varphi\colon A\to B$ be a
homomorphism of rings. Then for every element $\sigma\in \Sigma$
there exists a unique difference homomorphism $\Phi_{\sigma}\colon
A\to \Fun B$ such that the following diagram is commutative
$$
\xymatrix{
                                    & {\Fun B}\ar[d]^{\gamma_{\sigma}} \\
        A\ar[r]^{\varphi}\ar[ur]^{\Phi_\sigma} & B
}
$$
\end{proposition}
\begin{proof}
By the hypothesis the homomorphism $\Phi_\sigma$ satisfies the
property
$$
\Phi_\sigma(a)(\tau^{-1}\sigma)=(\tau\Phi_\sigma(a))(\sigma)=\varphi(\tau
a)
$$
whenever $a\in A$ and $\tau\in \Sigma$. Consequently, if
$\Phi_\sigma$ exists then it is unique. Define the mapping
$\Phi_\sigma$ by the following rule
$$
\Phi_\sigma(a)(\tau)=\varphi(\sigma\tau^{-1} a).
$$
It is clear that this mapping is a homomorphism. The following
calculation shows that this homomorphism is a difference one.
$$
(\nu
\Phi_\sigma(a))(\tau)=\Phi_\sigma(a)(\nu^{-1}\tau)=\varphi(\sigma\tau^{-1}\nu
a)=\Phi_\sigma(\nu a)(\tau).
$$
\end{proof}

The ring $\Fun B$ is an essential analogue of the Hurwitz series
ring. The elements of $\Fun B$ are the analogues of the Taylor
series. The homomorphisms $\Phi_\sigma$ are analogues of the Taylor
homomorphism for the Hurwitz series ring. Therefore, we shall call
these homomorphisms the Taylor homomorphisms at $\sigma$. The Taylor
homomorphism at the identity of the group will be called simpler the
Taylor homomorphism.

It should be noted that the set of all invariant elements of $\Fun
B$ can be identified with $B$. Namely, $B$ coincides with the set of
all constant functions. So, we suppose that $B$ is embedded in $\Fun
B$.

\section{The case of finite group}\label{sec4}

\subsection{Basic technique}\label{sec41}

From now we shall suppose that the group $\Sigma$ is finite. First
of all we shall prove more delicate technical results for the finite
group.

\begin{proposition}\label{techbasic}
Let $A$ be a difference ring, $\frak q$ be a pseudoprime ideal of
$A$, and $S$ be a multiplicatively closed subset in $A$. Then
\begin{enumerate}
\item Every minimal prime ideal containing $\frak q$ is
$\Sigma$-associated with $\frak q$.
\item The restriction of  $\pi$ onto $\Max A$ is a well-defined mapping $\pi\colon \Max A\to\PMax A$.
\item If $S\cap\frak q=\emptyset$ then $(S(\frak q))_{\Sigma}=\frak
q$.
\end{enumerate}
\end{proposition}
\begin{proof}
(1). Let $\frak p$ be a $\Sigma$-associated with $\frak q$ prime
ideal. Then $\frak q=\cap_\sigma \frak p^\sigma$. Now let $\frak p'$
be an arbitrary minimal prime ideal containing $\frak q$. Then
$\cap_\sigma \frak p^\sigma=\frak q\subseteq\frak p'$. Consequently,
$\frak p^\sigma\subseteq \frak p'$  for some $\sigma$ and, thus,
$\frak p^\sigma=\frak p'$.

(2). Let $\frak m$ be a maximal ideal and let $\frak q=\frak
m_{\Sigma}$. We shall show that $\frak q$ is a maximal difference
ideal. Since the mapping $\Spec A\to\PSpec A$ is surjective, it
suffices to show that every prime ideal containing $\frak q$
coincides with $\frak m^\sigma$ for some $\sigma$. Indeed, let
$\frak q\subseteq\frak p$. Then since $\frak q=\cap_\sigma \frak
m^\sigma$, we have $\frak m^\sigma\subseteq\frak p$  for some
$\sigma$. The desired result holds because $\frak m$ is maximal.

(3). By the hypothesis there is a prime ideal $\frak p$ such that
$\frak q\subseteq \frak p$ and $S\cap\frak p=\emptyset$. Then there
exists a minimal prime ideal $\frak p'$ with the same condition.
From the definition we have $S(\frak q)\subseteq S(\frak p')=\frak
p'$. Thus, the equality $\frak q=\frak p'_\Sigma$ follows from
condition~(1).
\end{proof}

Let $A$ be a difference ring, and $X$ be the pseudospectrum of $A$.
For any radical difference ideal $\frak t$ we define the closed
subset $V(\frak t)$ in $X$. Conversely, for every closed subset $Z$
we define the radical difference ideal $\cap_{\frak q\in Z}\frak q$.

\begin{proposition}
The mentioned mappings are inverse to each other bijections between
$\Rad^\Sigma A$ and $\{\,Z\subseteq X\mid Z=V(E)\,\}$. Suppose
additionally that any radical difference ideal in $A$ is an
intersection of finitely many prime ideals (for example $A$ is
Noetherian). Then a closed set is irreducible if and only if it
corresponds to a pseudoprime ideal.
\end{proposition}
\begin{proof}

The first statement follows from Proposition~\ref{pspecst}~(4). Let
us show that irreducible sets correspond to pseudoprime ideals.

Let $\frak q$ be a pseudoprime ideal and let $V(\frak q)=V(\frak
a)\cup V(\frak b)=Ì(\frak a\cap\frak b)$. Then $\frak
q\supseteq\frak a\cap\frak b$. Thus, either  $\frak q\supseteq\frak
a$, or $\frak q\supseteq\frak b$ (see
Proposition~\ref{pspectop}~(4)). Suppose that the first condition
holds. Then $V(\frak q)\subseteq V(\frak a)$. The other inclusion
holds because of the choice of $V(\frak a)$.

Conversely, let $\frak t$ be a radical difference ideal. Suppose
that $\frak t$ is not pseudoprime. Let $\frak p_1\ldots,\frak p_n$
be all minimal prime ideals containing $\frak t$. Then the action of
$\Sigma$ on this set is not transitive. Thus, the ideals
$$
\frak t_i=\bigcap_{\sigma\in\Sigma}\frak p_i^\sigma
$$
contain $\frak t$ and do not coincide with $\frak t$. Let $\frak
t_1,\ldots,\frak t_s$ be all different ideals among $\frak t_i$.
Then $\frak t =\cap_i \frak t_i$ is a nontrivial decomposition of
the ideal $\frak t$.
\end{proof}

\subsection{Inheriting of properties}\label{sec42}

Let $f\colon A\to B$ be a difference homomorphism of difference
rings. We shall consider the following pairs of properties:
\begin{description}
\item{({\bf A1}):} is a property of $f$, where $f$ is considered as a homomorphism
\item{({\bf A2}):} is a property of $f$, where $f$ is considered as a difference homomorphism
\end{description}
such that ({\bf A1}) implies ({\bf A2}). The idea is the following:
finding such pair of properties, we shall reduce the difference
problem to a non difference one.

The homomorphism $f\colon A\to B$ is said to have the going-up
property if for every chain of prime ideals $$\frak p_1\subseteq
\frak p_2\subseteq\ldots\subseteq\frak p_n$$ in $A$ and every chain
of prime ideals $$\frak q_1\subseteq \frak
q_2\subseteq\ldots\subseteq\frak q_m$$ in $B$ such that $0<m<n$ and
$\frak q^c_i=\frak p_i$ ($1\leqslant i\leqslant m$) the second chain
can be extended to a chain  $$\frak q_1\subseteq \frak
q_2\subseteq\ldots\subseteq\frak q_n$$ with condition $\frak
q^c_i=\frak p_i$ ($1\leqslant i\leqslant n$).

The homomorphism $f\colon A\to B$ is said to have the going-down
property if for every chain of prime ideals $$\frak p_1\supseteq
\frak p_2\supseteq\ldots\supseteq\frak p_n$$ in $A$ and every chain
of prime ideals $$\frak q_1\supseteq \frak
q_2\supseteq\ldots\supseteq\frak q_m$$ in $B$ such that $0<m<n$ and
$\frak q^c_i=\frak p_i$ ($1\leqslant i\leqslant m$), the second
chain can be extended to a chain $$\frak q_1\supseteq \frak
q_2\supseteq\ldots\supseteq\frak q_n$$ with condition $\frak
q^c_i=\frak p_i$ ($1\leqslant i\leqslant n$).

Let $f\colon A\to B$ be a difference homomorphism. This homomorphism
is said to have going-up (going-down) property for difference ideals
if the mentioned above properties hold for the chains of pseudoprime
ideals.

\begin{proposition}
For every difference homomorphism $f\colon A\to B$ the following
holds
\begin{enumerate}
\item In the following diagram
$$
\xymatrix{
    \Spec B\ar[r]^{f^*}\ar[d]^{\pi}&\Spec A\ar[d]^{\pi}\\
    \PSpec B\ar[r]^{f^*_\Sigma}&\PSpec A\\
}
$$
if $f^*$ is surjective, then $f^*_{\Sigma}$ is surjective.
\item $f$ has the going-up property $\Rightarrow$ $f$
has the going-up property for difference ideals.
\item $f$ has the going-down property $\Rightarrow$ $f$
has the going-down property for difference ideals.
\end{enumerate}
\end{proposition}
\begin{proof}
(1) This property follows from the fact that $\pi$ is surjective.

(2). Let $\frak q_1\subseteq\frak q_2$ be a chain of pseudoprime
ideals of $A$ and let $\frak q'_1$ be a pseudoprime ideal in $B$
contracting to $\frak q_1$. Consider a prime ideal $\frak p'_1$
$\Sigma$-associated with $\frak q'_1$. The contraction of $\frak
p'_1$ to $A$ will be denoted by $\frak p_1$. Then $\frak p_1$ will
be $\Sigma$-associated with $\frak q_1$. Let $\frak p_2$ be a prime
ideal $\Sigma$-associated with $\frak q_2$. Then $\cap_\sigma \frak
p^\sigma_1=\frak q_1\subseteq\frak p_2$. Thus, it follows
from~\cite[chapter~1, sec.~6, prop.~1.11(2)]{AM} that for some
$\sigma$ we have $\frak p^\sigma_1\subseteq\frak p_2$. Consider two
chains of prime ideals $\frak p^\sigma_1\subseteq\frak p_2$ in $A$
and $(\frak p'_1)^\sigma$ in $B$. From the going-up property there
exists a prime ideal  $\frak p'_2$ such that $(\frak
p'_1)^\sigma\subseteq\frak p'_2$ and $(\frak p'_2)^c=\frak p_2$.
Therefore, the ideal $(\frak p'_2)_\Sigma$ is the desired
pseudoprime ideal.

(3). Let $\frak q_1\supseteq\frak q_2$ be a chain of pseudoprime
ideals in $A$, and let $\frak q'_1$ be a pseudoprime ideal in $B$
contracting to $\frak q_1$. Let $\frak p'_1$ be a prime ideal
$\Sigma$-associated with $\frak q'_1$. Its contraction to $A$ will
be denoted by $\frak p_1$. Then $\frak p_1$ is $\Sigma$-associated
with $\frak q_1$. Let $\frak p$ be a prime ideal $\Sigma$-associated
with $\frak q_2$. Then  $\cap_\sigma \frak p^\sigma=\frak
q_2\subseteq\frak p_1$. Consequently, for some $\sigma$ we have
$\frak p^\sigma\subseteq \frak p$ (see.~\cite[chapter~1, sec.~6,
prop.~1.11(2)]{AM}). The going-down property guaranties that there
exists a prime ideal $\frak p'_2$ with conditions $\frak
p'_2\subseteq\frak p'_1$ and $(\frak p'_2)^c=\frak p^\sigma$. Then
the ideal  $(\frak p'_2)_\Sigma$ is the desired one.
\end{proof}

Since not for every multiplicatively closed set $S$ the fraction
ring is a difference ring we need to generalize the previous
proposition. Let $f\colon A\to B$ be a difference homomorphism and
let $X$ and $Y$ be subsets of the pseudospectra of $A$ and $B$,
respectively, such that $f^*_\Sigma(Y)\subseteq X$. We shall say
that that the going-up property holds for $f^*_\Sigma\colon Y\to X$
if for every chain of pseudoprime ideals $$\frak p_1\subseteq \frak
p_2\subseteq\ldots\subseteq\frak p_n$$ in $X$ and every chain of
pseudoprime ideals $$\frak q_1\subseteq \frak
q_2\subseteq\ldots\subseteq\frak q_m$$ in $Y$ such that $0<m<n$ and
$\frak q^c_i=\frak p_i$ ($1\leqslant i\leqslant m$), the second
chain can be extended to a chain  $$\frak q_1\subseteq \frak
q_2\subseteq\ldots\subseteq\frak q_n$$ in $Y$ with condition $\frak
q^c_i=\frak p_i$ ($1\leqslant i\leqslant n$).

We shall say that that the going-down property holds for
$f^*_\Sigma\colon Y\to X$ if for every chain of pseudoprime ideals
$$\frak p_1\supseteq \frak p_2\supseteq\ldots\supseteq\frak p_n$$ in
$X$ and every chain of pseudoprime ideals $$\frak q_1\supseteq \frak
q_2\supseteq\ldots\supseteq\frak q_m$$ in $Y$ such that $0<m<n$ and
$\frak q^c_i=\frak p_i$ ($1\leqslant i\leqslant m$), the second
chain can be extended to a chain  $$\frak q_1\supseteq \frak
q_2\supseteq\ldots\supseteq\frak q_n$$ in $Y$ with condition $\frak
q^c_i=\frak p_i$ ($1\leqslant i\leqslant n$). Now we shall prove
more delicate result.

\begin{proposition}\label{inher}
Let $f\colon A\to B$ be a difference homomorphism of difference
rings. The pseudospectrum of $A$ will be denoted by $X$ and the
pseudospectrum of $B$ by $Y$. Then the following holds:
\begin{enumerate}
\item Let for some $s\in A$ and $u\in B$ the mapping
$$
f^*\colon \Spec B_{su}\to \Spec A_s
$$
be surjective. Then the mapping  $f^*_\Sigma\colon Y_{f(s)u}\to X_s$
is surjective.
\item Let for some $s\in A$ the mapping
$$
f^*_s\colon \Spec B_s\to \Spec A_s
$$
have the going-up property. Then the mapping $f^*_\Sigma\colon
Y_{f(s)}\to X_s$ has the going-up property.
\item Let for some $s\in A$ and $u\in B$ the mapping
$$
f^*\colon \Spec B_{su}\to \Spec A_s
$$
have the going-down property. Then the mapping $f^*_\Sigma\colon
Y_{f(s)u}\to X_s$ has the going-down property.
\end{enumerate}
\end{proposition}
\begin{proof}

(1). Let $\frak q\in X_s$. Since $s\notin \frak q$ then there exists
a prime ideal $\frak p'$ such that  $\frak q\subseteq\frak p'$ and
$s\notin\frak p'$. Then there exists a minimal prime ideal $\frak p$
with this property. Consequently, $\frak p$ is $\Sigma$-associated
with $\frak q$. By the hypothesis there is a prime ideal  $\frak
p_1$ in $B$ not containing $f(s)u$ such that  $\frak p_1^c=\frak p$.
Therefore, the ideal $(\frak p_1)_\Sigma$  is the desired one.

(2). Let $\frak q_1\subseteq\frak q_2$ be a chain of pseudoprime
ideals of $A$ not containing $s$, and $\frak q'_1$ be a pseudoprime
ideal of $B$ not containing $f(s)$. Let  $\frak p'_1$ be a prime
ideal $\Sigma$-associated with $\frak q'_1$, and $\frak p_1$ is its
contraction to $A$. As in~(1) we shall find a prime ideal $\frak
p_2$ $\Sigma$-associated with $\frak q_2$ and not containing $s$.
Then $\cap_\sigma \frak p^\sigma_1=0\subseteq\frak p_2$. Thus, for
some $\sigma$ we have $\frak p^\sigma_1\subseteq \frak p_2$.
Consider the sequence of ideals  $\frak p^\sigma_1\subseteq \frak
p_2$ and ideal $(\frak p'_1)^\sigma$ contracting to $\frak
p^\sigma_1$. By the hypothesis there exists a prime ideal  $\frak
p'_2$ containing  $(\frak p'_1)^\sigma$ and contracting to  $\frak
p_2$. Then the ideal $(\frak p'_2)_\Sigma$ is the desired one.

(3). Let $\frak q_1\supseteq\frak q_2$  be a chain of pseudoprime
ideals in $A$ not containing $s$ and $\frak q'_1$  is a pseudoprime
ideal in $B$ contracting to $\frak q_1$. As in~(2) we shall find a
prime ideal $\frak p'_1$ $\sigma$-associated with $\frak q'_1$ and
not containing $f(s)u$. Its contraction will be denoted by $\frak
p_1$. Let $\frak p_2$ be a prime ideal $\Sigma$-associated with
$\frak q_2$. Then $\cap_\sigma \frak p_2^\sigma=0\subseteq\frak
p_1$. Thus, for some $\sigma$ we have $\frak
p_2^\sigma\subseteq\frak p_1$. By the hypothesis for the chain
$\frak p_1\supseteq\frak p_2^\sigma$ and the ideal $\frak p'_1$
there is a prime ideal $\frak p'_2$ laying in $\frak p'_1$ and
contracting to $\frak p_2^\sigma$. Then the ideal $(\frak
p'_2)_\Sigma$ is the desired one.
\end{proof}

\begin{example}
Let $\Sigma=\mathbb Z/2\mathbb Z$, where $\sigma=1$ is the nonzero
element of the group, and let $C$ be an algebraically closed field.
Let $A=C[x]$, where $\sigma$ coincides with the identity mapping on
$A$. Now consider the ring $B=C[t]$, where $\sigma(t)=-t$. There is
a difference embedding $\varphi\colon A\to B$ such that  $x\mapsto
t^2$. So, we can identify $A$ with the subring $C[t^2]$ in $B$.

Let $\dcspec B$ and $\dcspec A$ be the sets of all prime difference
ideals of the rings $B$ and $A$, respectively. It is clear that
$\dcspec B=\{0\}$ consists of one single point and $\dcspec A=\Spec
A$. The contraction mapping
$$
\varphi^*\colon \dcspec B\to\dcspec A
$$
maps the zero ideal to the zero ideal. We see that $\dcspec B$ is
dense in $\dcspec A$ but does not contain an open in its closure.
So, $\dcspec B$ is very poor.

Now let us show what will happen if we use pseudoprime ideals
instead of prime ones. It is clear that $\PSpec A=\Spec A$. Let us
describe $\PSpec B$. Consider the mapping $\pi\colon\Spec B\to
\PSpec B$. Every maximal ideal of $B$ is of the form $(t-a)$, then
$(t-a)_\Sigma=(t^2-a^2)$ is pseudoprime if $a\neq 0$. Therefore, the
set of all pseudoprime ideals is the following
$$
\PSpec B=\{0\}\cup\{(t)\}\cup\{\,(t^2-a)\mid 0\neq a\in C\,\}.
$$
We can identify pseudomaximal spectrum with an affine line $C$ by
the rule $(t^2-a)\mapsto a$ and $(t)\mapsto 0$. Now consider the
contraction mapping
$$
\varphi^*_{\Sigma}\colon \PSpec B\to \PSpec A.
$$
As we can see $\varphi^*_\Sigma(t^2-a)=(x-a)$ and
$\varphi^*_\sigma(t)=(x)$. Identifying pseudomaximal spectrum of $A$
with $C$ by the rule $(x-a)\mapsto a$, we see that the mapping
$\varphi^*_\Sigma\colon \PMax B\to \PMax A$ coincides with the
identity mapping. It is easy to see that the homomorphism
$\varphi\colon A\to B$ has the going-up and going-down properties.
Therefore, it has the going-up and the going-down properties for
difference ideals. But this is obvious from the discussion above.
Consequently, the mapping $\varphi^*_\Sigma$ is a homeomorphism
between $\PSpec A$ and $\PSpec B$.
\end{example}

\subsection{Pseudofields}\label{sec43}

An absolutely flat simple difference ring will be called a
pseudofield.

\begin{proposition}
For every pseudofield $A$ the group $\Sigma$ is transitively acting
on $\Max A$. Moreover, as a commutative ring  $A$ is isomorphic to
$K^n$, where $n$ is the number of all maximal ideals in $A$ and $K$
is isomorphic to $A/\frak m$, where $\frak m$ is a maximal ideal in
$A$.
\end{proposition}
\begin{proof}

Let $\frak m$ be a prime ideal of $A$. By the hypothesis this ideal
is simultaneously maximal and minimal (see.~\cite[chapter~3,
ex.~11]{AM}). Then $\cap_\sigma\frak m^\sigma$  is a difference
ideal and, thus, equals zero. Let $\frak n$ be an arbitrary prime
ideal of $A$ then $\cap_\sigma\frak m^\sigma=0\subseteq \frak n$.
Consequently, $\frak n=\frak m^\sigma$  for some $\sigma$, i.~e.,
$\Sigma$ acts transitively on $\Max A$. Let $\frak m_1,\ldots,\frak
m_n$  be the set of all maximal ideals of $A$. Then it follows
from~\cite[chapter~1, sec.~6, prop.~1.10]{AM} that $A$ is isomorphic
to $\prod_i A/\frak m_i$. Since every element of $\Sigma$ is an
isomorphism then for every $\sigma$ the field  $A/\frak m$ is
isomorphic to $A/\frak m^\sigma$.
\end{proof}

\begin{proposition}
Let $A$ be a difference ring and $\frak q$ be its difference ideal.
The ideal is pseudomaximal if and only if $A/\frak q$ is
pseudofield. In other words every simple difference ring is
absolutely flat.
\end{proposition}
\begin{proof}
If $A/\frak q$ is a pseudofield then $\frak q$ is a maximal
difference ideal and, hence, pseudomaximal. Conversely, let $\frak
q$ be pseudomaximal and $\frak m$ is a maximal ideal containing
$\frak q$. Since $\frak q$ is a maximal difference ideal, then
$\frak m$ is $\Sigma$-associated with $\frak q$. Hence, $\frak
q=\cap_\sigma \frak m^\sigma$. And it follows from~\cite[chapter~1,
sec.~6, prop.~1.10]{AM} that $A/\frak q=\prod_\sigma A/\frak
m^\sigma$.
\end{proof}

As we see a simple difference ring and a pseudofield are the same
notions. Note that the ring $\Fun A$ is a pseudofield if and only if
$A$ is a field. We shall introduce the notion of difference closed
pseudofield. Let $A$ be a pseudofield. Consider the ring of
difference polynomials $A\{y_1,\ldots,y_n\}$. Let $E\subseteq
A\{y_1,\ldots,y_n\}$ be an arbitrary subset. The set of all common
zeros of $E$ in $A^n$ will be denoted by $V(E)$. Conversely, let
$X\subseteq A^n$ be an arbitrary subset. The set of all polynomials
vanishing on $X$ will be denoted by $I(X)$. It is clear that for any
difference ideal $\frak a\subseteq A\{y_1,\ldots,y_n\}$ we have
$\frak r(\frak a)\subseteq I(V(\frak a))$. A pseudofield $A$ will be
said to be a difference closed pseudofield if for every $n$ and
every difference ideal $\frak a\subseteq A\{y_1,\ldots,y_n\}$ there
is the equality $\frak r(\frak a)= I(V(\frak a))$.

\begin{proposition}\label{weakcon}
If $A$ is a difference closed pseudofield, then every difference
finitely generated over $A$ pseudofield coincides with $A$.
\end{proposition}
\begin{proof}
Every difference finitely generated over $A$ pseudofield can be
presented as  $A\{y_1,\ldots,y_n\}/\frak q$, where $\frak q$ is a
pseudomaximal ideal. It is easy to see that the ideal $\frak q$ is
of the form $I(a)$ for some $a\in A^n$. Hence, $\frak
q=[y_1-a_1,\ldots,y_n-a_n]$. Therefore, $A\{y_1,\ldots,y_n\}/\frak
q$ coincides with $A$.
\end{proof}

\begin{proposition}\label{funcon}
A pseudofield $\Fun K$ is difference closed if and only if $K$ is
algebraically closed.
\end{proposition}
\begin{proof}
Let $\Fun K$ be difference closed. We recall that $K$ can be
embedded into $\Fun K$ as the subring of the constants. Consider the
ring
$$
R=\Fun K\{y\}/(\ldots,\sigma y-y,\ldots)_{\sigma\in\Sigma}.
$$
As a commutative ring it is isomorphic to $\Fun K[y]$. Let $f$ be a
polynomial in one variable with coefficients in $K$. The ideal
$(f(y))$ is a nontrivial ideal in $\Fun K[y]$. Moreover, since
$f(y)$ is an invariant element, the mentioned ideal is difference.
Consequently, $B=R/(f(y))$ is a nontrivial difference ring. Let
$\frak m$ be a pseudomaximal ideal in $B$. Then the pseudofield
$B/\frak m$ coincides with $\Fun K$ because of the previous
proposition. Denote the image of the element $y$ in $\Fun K$ by $t$.
By the definition $f(t)=0$ and $t$ is invariant. Thus, $t$ is in
$K$. So, $K$ is algebraically closed.

Conversely, let $K$ be an algebraically closed field. Let $\frak a$
be an arbitrary difference ideal in $\Fun K\{y_1,\ldots,y_n\}$.
Consider the algebra
$$
B=\Fun K\{y_1,\ldots,y_n\}/\frak a.
$$
We shall show that for every element $s\in B$ not belonging to the
nilradical there is a difference homomorphism $f\colon B\to\Fun K$
over $\Fun K$ such that $f(s)\neq0$. From Proposition~\ref{taylor}
it suffices to find a homomorphism $\psi\colon B\to K$ such that for
some $\sigma$ the following diagram is commutative
$$
\xymatrix@R=10pt{
    B\ar[rd]^{\psi}&\\
    \Fun K\ar[r]^{\gamma_\sigma}\ar[u]&K
}
$$
Indeed, consider the ring $B_s$ and let $\frak n$ be a maximal ideal
of $B_s$. Then $B_s/\frak n$  is a finitely generated algebra over
$K$ and is a field. Therefore, $B_s/\frak n$ coincides with $K$ (see
the Hilbert Nullstellensatz) and this quotient mapping gives us the
homomorphism $\psi\colon B\to K$. Let $\frak m = \Fun K\cap \frak
n$, then $\frak m$ coincides with the ideal $\ker \gamma_\sigma$ for
some $\sigma\in \Sigma$. So, the restriction of $\psi$ onto $\Fun K$
coincides with $\gamma_\sigma$.
\end{proof}

\begin{proposition}\label{singdef}
Let $A$ be a pseudofield. Suppose that every difference generated
over $A$ by one single element pseudofield coincides with $A$. Then
the Taylor homomorphism is an isomorphism between $A$ and $\Fun K$,
where $K=A/\frak m$  for every maximal ideal $\frak m$ of $A$.
\end{proposition}
\begin{proof}
Let $\frak m$ be a maximal ideal of $A$. Consider the field
$K=A/\frak m$ and define the ring $\Fun K$. It follows from
Proposition~\ref{taylor} that there exists a difference homomorphism
$\Phi\colon A\to \Fun K$ for the quotient homomorphism $\pi\colon
A\to K$.
$$
\xymatrix{
    &\Fun K\ar[d]^{\gamma_e}\\
    A\ar[ur]^{\Phi}\ar[r]^{\pi}&K
}
$$

Since $A$ is a simple difference ring $\Phi$ is injective. Let us
show that $\Phi$ is surjective. Assume that contrary holds and there
is an element $\eta\in\Fun K\setminus A$. Consider  the ring
$A\{y\}=A[\ldots,\sigma y,\ldots]$ and its quotient ring
$K[\ldots,\sigma y,\ldots]$. The ideal $(\ldots,\sigma
y-\eta(\sigma),\ldots)$ is maximal in the latter ring. This ideal
contracts to the maximal ideal  $\frak m'$ in $A\{y\}$. It follows
from Proposition~\ref{techbasic} that the ideal $\frak n=\frak
m'_\Sigma$ is pseudomaximal. So,  $A\{y\}/\frak n$ is a pseudofield
difference generated over $A$ by one singly element. Thus,
$A\{y\}/\frak n$ coincides with $A$. On the other hand, the
following is a homomorphism
$$
\varphi\colon A\{y\}/\frak n\to A\{y\}/\frak m'= K[\ldots,\sigma
y,\ldots]/(\ldots,\sigma y-\eta(\sigma),\ldots)=K.
$$
The restriction of this homomorphism to $A$ coincides with the
quotient homomorphism $\pi$. Proposition~\ref{taylor} guaranties
that there is a difference embedding $\Psi\colon A\{y\}/\frak n\to
\Fun K$. It follows from the uniqueness of the Taylor homomorphism
that the restriction of the last mapping to $A$ coincides with
$\Phi$.
$$
\xymatrix{
    A\{y\}/\frak n\ar[r]^{\Psi}\ar[rd]^{\varphi}&\Fun K\ar[d]^{\gamma_e}\\
    A\ar[r]^{\pi}\ar[u]^{Id}&K
}
$$
From the definition we have $\Psi(y)(\sigma)=\eta(\sigma)$.
Consequently, $\Psi(y)=\eta$ and, thus, the image of pseudofield
$A\{y\}/\frak n$ contains $\eta$. Form the other hand the image
coincides with $A$, contradiction.
\end{proof}

The following theorem is a corollary of the previous propositions.

\begin{theorem}\label{equtheor}
Let $A$ be a pseudofield, then the following conditions are
equivalent:
\begin{enumerate}
\item $A$ is difference closed.
\item Every difference finitely generated over $A$ pseudofield coincides with $A$.
\item Every pseudofield generated over $A$ by one single element coincides with $A$.
\item The pseudofield $A$ is isomorphic to $\Fun K$, where $K$ is an algebraically closed field.
\end{enumerate}
\end{theorem}
\begin{proof}
(1)$\Rightarrow$(2). It follows from Proposition~\ref{weakcon}.

(2)$\Rightarrow$(3). Is trivial.

(3)$\Rightarrow$(4). By Proposition~\ref{singdef}, it follows that
the ring $A$ is isomorphic to $\Fun K$. We only need to show that
$K$ is algebraically closed (see Proposition~\ref{funcon}). For that
we shall repeat the first half of the proof of
Proposition~\ref{funcon}.

We know that every pseudofield difference generated by one single
element over $\Fun K$ coincides with $\Fun K$. Let us recall that
$K$ can be embedded into $\Fun K$ as the subring of the constants.
Consider the ring
$$
R=\Fun K\{y\}/(\ldots,\sigma y-y,\ldots)_{\sigma\in\Sigma}.
$$
As a commutative ring it is isomorphic to  $\Fun K[y]$. Let $f$ be a
polynomial in one variable with coefficients in $K$. The ideal
$(f(y))$ is a nontrivial ideal in $\Fun K[y]$. Moreover, since
$f(y)$ is an invariant element the mentioned ideal is difference.
Let $\frak m$ be a pseudomaximal ideal in $B$, then the pseudofield
$B/\frak m$ coincides with $\Fun K$. The image of $y$ in $\Fun K$
will be denoted by $t$. By the definition we have $f(t)=0$  and $t$
is an invariant element. Thus, $t$ is in $K$. Therefore, the field
$K$ is algebraically closed.

(4)$\Rightarrow$(1). It follows from Proposition~\ref{funcon}
\end{proof}

\begin{example}
Consider the field of complex numbers $\mathbb C$ and its
automorphism $\sigma$ (the complex conjugation). This pair can be
regarded as a difference ring with a group $\Sigma=\mathbb
Z/2\mathbb Z$. Let $\mathbb C[x]$ be the ring of polynomials over
$\mathbb C$ and automorphism $\sigma$ is acting as follows
$\sigma(f(x))=\overline{f}(-x)$. Then the ideal  $(x^2-1)$  is a
difference ideal. Consider the ring $A=\mathbb C[x]/(x^2-1)$. As a
commutative ring it can be presented as follows
$$
\mathbb C[x]/(x^2-1)=\mathbb C[x]/(x-1)\times\mathbb
C[x]/(x+1)=\mathbb C\times \mathbb C.
$$
Under this mapping an element $c\in \mathbb C$ maps to  $(c,c)$ and
$x$ maps to $(1,-1)$. The automorphism acts as follows
$(a,b)\mapsto(\overline{b},\overline{a})$. Consider the projection
of $A$ onto its first factor. For this homomorphism there is the
Taylor homomorphism $A\to \Fun \mathbb C$. As a commutative ring the
ring $\Fun \mathbb C$ coincides with $\mathbb C\times \mathbb C$.
Automorphism acts as follows $(a,b)\mapsto(b,a)$. The Taylor
homomorphism is defined by the following rule $a+bx\mapsto
(a+b,\overline{a}-\overline{b})$.

Now we have two homomorphisms: the first one is $f\colon A\to
\mathbb C\times\mathbb C$ and is defined by the rule
$$
a+bx\mapsto (a+b,a-b)
$$
and the second one is $g\colon A\to \mathbb C\times \mathbb C$ and
is defined by the rule
$$
a+bx\mapsto (a+b,\overline{a}-\overline{b}).
$$
Then composition  $g\circ f^{-1}$ acts as follows $(a,b)\mapsto
(a,\overline{b})$.

So, pseudofield $A$ is difference closed. Moreover, the homomorphism
$g\circ f^{-1}$ transforms the initial action of $\sigma$ into more
simple one.
\end{example}

Let $A$ be a pseudofield and $\frak m$ is its maximal ideal. Then
the residue field of $\frak m$ will be denoted by $K$, i.~e.,
$K=A/\frak m$. Let $L$ be the algebraical closure of $K$. The
pseudofield $\Fun L$ will be denoted by $\overline{A}$. Let
$\varphi\colon A\to L$ be the composition of the quotient morphism
and the natural embedding of $K$ to $L$. Let $\Phi\colon A\to
\overline{A}$ be the Taylor homomorphism corresponding to $\varphi$.
We know that $\overline{A}$ is difference closed. Let us show that
$\overline{A}$ is a minimal difference closed pseudofield containing
$A$.

\begin{proposition}\label{difclosemin}
Let $D$ be a difference closed pseudofield such that $A\subseteq
D\subseteq \overline{A}$. Then  $D=\overline{A}$.
\end{proposition}
\begin{proof}
Consider the sequence of rings $A\subseteq D\subseteq \overline{A}$.
Let $\frak m$ be a maximal ideal of $\overline{A}$. Then we have the
following sequence of fields
$$
A/A\cap\frak m\subseteq D/D\cap\frak m\subseteq \overline{A}/\frak
m.
$$
Since $D$ is difference closed, it follows from
Theorem~\ref{equtheor} that the field $D/D\cap\frak m$ coincides
with $L=\overline{A}/\frak m$. Now consider the composition of $D\to
D/D\cap\frak m$ and  $D/D\cap\frak m\to L$ and let $\Psi\colon D\to
L$ be the corresponding Taylor homomorphism. It follows from the
uniqueness of the Taylor homomorphism that $\Psi$ coincides with the
initial embedding of $D$ to $\overline{A}$. So, $D$ satisfies the
condition of Proposition~\ref{singdef}.
\end{proof}

\begin{proposition}\label{difclosuni}
Let $B$ be a difference closed pseudofield  containing $A$. Then
there exists an embedding of  $\overline{A}$ to $B$ over $A$.
\end{proposition}
\begin{proof}

On the following diagram arrows present the embeddings of $A$ to
$\overline{A}$ and to $B$ respectively:
$$
\xymatrix@R=5pt@C=15pt{
    \overline{A}&&B\\
    &A\ar[ul]\ar[ur]&
}
$$

Let $\frak m$ be a maximal ideal in $B$. Then it contracts to a
maximal ideal $\frak m^c$ in $A$. Since $A$ is an absolutely flat
ring there exists an ideal $\frak n$ in  $\overline{A}$ contracting
to $\frak m^c$ (see~\cite[chapter~3, ex.~29, ex.~30]{AM}). So, we
have
$$
\xymatrix@R=10pt@C=15pt{
    \overline{A}/\frak n&&B\ar[d]\\
    &A/\frak m^c\ar[ul]\ar[r]&B/\frak m
}
$$

By the definition the field $\overline{A}/\frak n$  is the algebraic
closure of $A/\frak m^c$ and $B/\frak m$ is algebraically closed
(Theorem~\ref{equtheor}). Therefore, there exists an embedding of
$\overline{A}/\frak n$ to $B/\frak m$.
$$
\xymatrix@R=10pt@C=15pt{
    \overline{A}\ar[d]&&B\ar[d]\\
    {\overline{A}/\frak n}\ar@/^1pc/[rr]&A/\frak m^c\ar[l]\ar[r]&B/\frak m
}
$$

So, there is a homomorphism $\overline{A}\to B/\frak m$. Then
Proposition~\ref{taylor} guaranties that there is a difference
homomorphism $\varphi$ such that the following diagram is
commutative
$$
\xymatrix@R=10pt@C=15pt{
    \overline{A}\ar[d]\ar@/^/[rrd]\ar[rr]^{\varphi}&&B\ar[d]\\
    {\overline{A}/\frak n}\ar@/^1pc/[rr]&A/\frak m^c\ar[l]\ar[r]&B/\frak m
}
$$

The restriction of $\varphi$ onto $A$ coincides with the Taylor
homomorphism for the mapping $A\to B/\frak m$. It follows from the
uniqueness that the Taylor homomorphism coincides with the initial
embedding of $A$ to $B$.
\end{proof}

\begin{example}
Let $\Sigma=\mathbb Z/2\mathbb Z$ and $\sigma=1$ be a nonzero
element of $\Sigma$. Consider the field $\mathbb C(t)$, where $t$ is
a transcendental element over  $\mathbb C$. We assume that action of
$\Sigma$ is trivial on $\mathbb C(t)$. Consider the following system
of difference equations
$$
\left\{
\begin{aligned}
\sigma x&=-x,\\
x^2&=t.
\end{aligned}
\right.
$$

Let $L$ be the algebraical closure of the field $\mathbb C(t)$. Then
the difference closure of $\mathbb C(t)$ coincides with $\Fun L$.
From the definition we have $\Fun L=L\times L$, where the first
factor corresponds to $0$ and the second one to $1$ in  $\mathbb
Z/2\mathbb Z$. Then our system has the two solutions
$(\sqrt{t},-\sqrt{t})$ and $(-\sqrt{t},\sqrt{t})$.

Moreover, we are able to construct the field containing the
solutions of this system. Consider the ring of polynomials $\mathbb
C(t)[x]$, where $\sigma x=-x$. Then the ideal $(x^2-t)$ is a maximal
difference ideal. Define
$$
D=\mathbb C(t)[x]/(x^2-t).
$$

By the definition $D$ is a minimal field containing solutions of the
system. From the other hand, Proposition~\ref{taylor} guaranties
that $D$ can be embedded into the difference closure of $\mathbb
C(t)$.
\end{example}

\begin{example}
Consider a ring $A=\mathbb C\times \mathbb C$ and a group
$\Sigma=\mathbb Z/4\mathbb Z$. Let $\sigma=1$ be a generator of
$\Sigma$. Let $\Sigma$ act on $A$ by the following rule
$\sigma(a,b)=(b,\overline{a})$. Then $\sigma$ is an automorphism of
fourth order. Consider the projection of $A$ onto the first factor.
Then there exists a homomorphism $\Phi\colon A\to \Fun \mathbb C$
such that the following diagram is commutative
$$
\xymatrix{
    &\Fun \mathbb C\ar[d]^{\gamma_e}\\
    A\ar[r]^{\pi}\ar[ru]^{\Phi}&\mathbb C
}
$$
where $\pi$ is the projection onto the first factor of $A$.

The pseudofield $\Fun \mathbb C$ is of the following form  $\mathbb
C_0\times\mathbb C_1\times\mathbb C_2\times\mathbb C_3$, where
$\mathbb C_i$ is a field $\mathbb C$ over the point $i$ of $\Sigma$.
Using this notation, the homomorphism $\gamma_e$ coincides with the
projection onto the first factor.  The element $\sigma$ acts on
$\Fun \mathbb C$ by the right transaction. The Taylor homomorphism
is defined by the rule $(a,b)\mapsto
(a,\overline{b},\overline{a},b)$.

Consider the embedding of $\mathbb C$ into $A$ by the rule $c\mapsto
(c,c)$ and the embedding into $\Fun \mathbb C$ by the rule
$(c,\overline{c},\overline{c},c)$. These both embeddings induce the
structure of  a $\mathbb C$-algebra. Since dimensions of $A$ and
$\Fun \mathbb C$ equal $2$ and $4$, respectively, $\Fun \mathbb C$
is generated by one single element over $A$. We shall find this
element explicitly. Consider the element  $x=(i,i,i,i)$ of $\Fun A$.
This element does not belong to $A$, therefore, $\Fun \mathbb
C=A\{x\}$. We have the following relations on the element $x$:
$\sigma x=x$ and $x^2+1=0$. Comparing the dimensions, we get $\Fun
\mathbb C =A\{y\}/[\sigma x-x, x^2+1]$.
\end{example}

\begin{example}
Let $\mathbb C$ be the field of complex numbers considered as a
difference ring over $\Sigma=\mathbb Z/2\mathbb Z$ and let
$\sigma=1$ be the nonzero element of the group. Then the system of
equations
$$
\left\{
\begin{aligned}
x\sigma x&=0,\\
x+\sigma x&=1
\end{aligned}
\right.
$$
has no solutions in every difference overfield containing $\mathbb
C$. But the ideal
$$
[x\sigma x,x+\sigma x-1]
$$
of the ring  $\mathbb C\{x\}$ is not trivial. Therefore, the system
has solutions in the difference closure of $\mathbb C$. The closure
coincides with $\Fun \mathbb C$. Namely, $\Fun \mathbb C=\mathbb
C\times\mathbb C$, where the first factor corresponds to zero and
the second one to the element $\sigma$. Then the solutions are
$(1,0)$ and $(0,1)$.
\end{example}

\begin{example}
Let $U$ be an open subset in the complex plane $\mathbb C$ and let
$\Sigma$ be a finite group of automorphisms of $U$. The ring of all
holomorphic functions in $U$ will be denoted by $A$. Then $A$ is a
$\Sigma$-algebra with respect to the action
$$
(\sigma\varphi)(z)=\varphi(\sigma^{-1} z).
$$
The difference closure of $\mathbb C$ is $\Fun \mathbb C$. Consider
an arbitrary point $x\in U$, then  there is a substitution
homomorphism $\psi_x\colon A\to \mathbb C$ such that $f\mapsto
f(x)$. Proposition~\ref{taylor} says that there exists the
corresponding Taylor homomorphism $\Psi_x\colon A\to \Fun \mathbb
C$.

Let us show the geometrical sense of this mapping. Consider the
orbit of the point $x$ and denote it by $\Sigma x$. Then there is a
natural mapping $\Sigma\to \Sigma x$ by $\sigma\mapsto \sigma x$.
Then for every function $\varphi\in A$ the composition of
$\Sigma\to\Sigma x$ and $\varphi|_{\Sigma x}\colon\Sigma x\to
\mathbb C$ coincides with the mapping $\Psi_x(\varphi)$. So, the
Taylor homomorphism $\Psi_x$ is just the restriction onto the orbit
of the given element $x$.
\end{example}

\begin{proposition}
Let $A$ be a pseudofield. Then the following conditions are
equivalent:
\begin{enumerate}
\item $A$ is difference closed
\item For every $n$ and every set $E\subseteq
A\{y_1,\ldots,y_n\}$ if there is a common zero for $E$ in $B^n$,
where $B$ is a pseudofield containing $A$, then there is a common
zero in $A^n$.
\item For every $n$, every set $E\subseteq
A\{y_1,\ldots,y_n\}$ and every finite set $W\subseteq
A\{y_1,\ldots,y_n\}$ if there is a common zero $b$ for $E$ in $B^n$,
where $B$ is a pseudofield containing $A$, such that no element of
$W$ vanishes on $b$, then there is a common zero for $E$ in $A^n$
such that no element of $W$ vanishes on it.
\end{enumerate}
\end{proposition}
\begin{proof}
(1)$\Rightarrow$(3). First of all we shall reduce the problem to the
case $|W|=1$. Let $b=(b_1,\ldots,b_n)\in B^n$ be the desired common
zero. The pseudofield $B$ can be embedded into its difference
closure $\overline{B}$. As we know $\overline{B}$ coincides with a
finite product of fields. Consider substitution homomorphism
$$
A\{y_1,\ldots,y_n\}\to B \to \overline{B}
$$
The composition of these two mappings we shall define by $\phi$. For
every element $w_i\in W$ we have $\phi(w_i)\neq0$. Thus for some
$\sigma_i$ we have $\phi(w_i)(\sigma_i)\neq 0$. By the definition of
$\Sigma$-action on $\overline{B}$, it follows that there is an
element $\tau_i\in \Sigma$ such that $\phi(\tau_i w_i)(e)\neq 0$.
(Actually, we know that $\tau_i=\sigma_i^{-1}$.) So,
$$
\phi\left(\prod_{i=1}^n\tau_iw_i\right)(e)\neq 0.
$$
Consider the polynomial $w=\prod_{i=1}^n\tau_iw_i$. It follows from
the definition that $\phi(w)\neq 0$. Moreover, the ring
$(A\{y_1,\ldots,y_n\}/[E])_w$ is not a zero ring.

Since $A$ is difference closed then it is of the form $A=\Fun K$,
where $K$ is algebraically closed. And there are homomorphisms
$\gamma_\sigma\colon A\to K$. Consider an arbitrary maximal ideal
$\frak n$ in $D=(A\{y_1,\ldots,y_n\}/[E])_w$ and let $\frak m$ be
its contraction to $A$. Then the field $D/\frak n$ is a finitely
generated algebra over $A/\frak m$. For the ideal $\frak m$ there is
a homomorphism $\gamma_\sigma$ such that $\frak m=\ker
\gamma_\sigma$. So, we have $A/\frak m= K$. Since $K$ is
algebraically closed, there is an embedding $D\to K$. So, we have
the following commutative diagram
$$
\xymatrix@R=20pt@C=20pt{
    A\{y_1,\ldots,y_n\}\ar[d]\ar[rd]^-{\phi}&\\
    (A\{y_1,\ldots,y_n\}/[E])_w\ar[r]&K
}
$$
such that $\phi|_A=\gamma_\sigma$. By Proposition~\ref{taylor}, it
follows that there exists a difference homomorphism $\varphi\colon
A\{y_1,\ldots,y_n\}\to A$ such that the following diagram is
commutative
$$
\xymatrix@R=20pt@C=20pt{
    A\{y_1,\ldots,y_n\}\ar[d]\ar[rd]^-{\phi}\ar[r]^-{\varphi}&A\ar[d]^{\gamma_\sigma}\\
    (A\{y_1,\ldots,y_n\}/[E])_w\ar[r]&K
}
$$
So, $\varphi$ is a difference homomorphism over $A$. The images of
$y_i$ give us the desired common zero in $A^n$.

(3)$\Rightarrow$(2). Is trivial.

(2)$\Rightarrow$(1). Let us show that every pseudofield difference
finitely generated over $A$ coincides with $A$. Let $B$ be a
pseudofield difference finitely generated over $A$. Then it can be
presented in the following form
$$
B=A\{y_1,\ldots,y_n\}/\frak m,
$$
where $\frak m$ is a pseudomaximal ideal. Then this ideal has a
common zero in $B^n$, $(y_1,\ldots,y_n)$ say. Consequently, there is
a common zero
$$
(a_1,\ldots,a_n)\in A^n.
$$
Consider a substitution homomorphism $A\{y_1,\ldots,y_n\}\to A$ by
the rule $y_i\mapsto a_i$. Then all elements of $\frak m$ maps to
zero. So, there is a difference homomorphism $B\to A$. Thus, $B$
coincides with $A$.
\end{proof}

\begin{proposition}
Let $A$ be a difference pseudofield with the residue field $K$ and
let $A[\Sigma]$ be the ring of difference operators on $A$. Then the
ring of difference operators is completely reducible and there is a
decomposition
$$
A[\Sigma]=A\oplus A\oplus\ldots\oplus A,
$$
where the number of summands is equal to size of the group $\Sigma$.
Moreover, we have
$$
A[\Sigma]=M_n(K),
$$
where $n=|\Sigma|$
\end{proposition}
\begin{proof}
Let us define the following module
$$
A_\tau = \{\,\sum_{\sigma} a_\sigma \delta_\sigma \sigma\tau\,\},
$$
where $\delta_\sigma$ is the indicator of the point $\sigma$. Using
the fact that $A=\Fun K$, we see that
$$
A[\Sigma]=\mathop{\oplus}_{\sigma\in \Sigma}A_\sigma.
$$
And moreover, every module $A_\sigma$ is isomorphic to $A$ as a
difference module. It follows from the equality
$$
A[\Sigma]=\Hom_{A[\Sigma]}(A[\Sigma],A[\Sigma])=M_n(\Hom_{A[\Sigma]}(A,A))
$$
that $A[\Sigma]=M_n(K)$.
\end{proof}

\begin{remark}
It follows from the previous proposition  that every difference
module over a difference closed pseudofield is free. Moreover, for
every such module there is a basis consisting of $\Sigma$-invariant
elements.
\end{remark}

\subsection{difference finitely generated algebras}\label{sec44}

The section is devoted to different technical conditions on
difference finitely generated algebras.

\begin{lemma}\label{lemma1}
Let $A$ be a ring with finitely many minimal prime ideals. Then
there exists an element $s\in A$ such that there is only one minimal
prime ideal in $A_s$.
\end{lemma}
\begin{proof}

Let $\frak p_1,\ldots,\frak p_n$ be all minimal prime ideal of $A$.
Then it follows from~\cite[chapter~1, sec.~6, prop.~1.11(II)]{AM}
that there exists an element $s$ such that
$$
s\in\bigcap_{i=2}^n\frak p_i\setminus\frak p_1.
$$
Then there is only one minimal prime ideal in $A_s$ and this ideal
corresponds to $\frak p_1$.
\end{proof}

\begin{lemma}\label{lemma2}
Let $A\subseteq B$ be rings such that $A$ is an integral domain and
$B$ is finitely generated over $A$. Then there exists an element
$s\in A$ with the following property. For any algebraically closed
field $L$ every homomorphism $A_s\to L$ can be extended to a
homomorphism $B_s\to L$.
\end{lemma}
\begin{proof}

Let $S=A\setminus 0$, consider the ring $S^{-1}B$. This ring is a
finitely generated algebra over the field $S^{-1}A$. Then there are
finitely many minimal prime ideals in $S^{-1}B$. These ideals
correspond to the ideals in $B$ contracting to $0$. Let $\frak p$ be
one of them. Consider the rings $A\subseteq B/\frak p$. It follows
from~\cite[chapter~5, ex.~21]{AM}  that there exists an element
$s\in A$ with the following property. For every algebraically closed
field $L$ every homomorphism  $A_s\to L$ can be extended to a
homomorphism $(B/\frak p)_s\to L$. Considering the composition of
the last one with $B_s\to (B/\frak p)_s$, we extend the initial
homomorphism to the homomorphism $B_s\to L$.
\end{proof}

\begin{lemma}\label{lemma3}
Let  $A\subseteq B$ be rings such that $A$ is an integral domain and
$B$ is finitely generated over $A$. Then there exists an element
$s\in A$ such that the corresponding mapping $\Spec B_s\to\Spec A_s$
is surjective.
\end{lemma}
\begin{proof}

From the previous lemma we find an element $s$. Let $\frak p$ be a
prime ideal in $A$ not containing $s$. The residue field of $\frak
p$ will be denoted by $K$. Let $L$ denote the algebraic closure of
$K$. The composition of the mappings $A\to K$ and $K\to L$ will be
denoted by $\varphi\colon A\to L$. By the definition we have
$\varphi(s)\neq 0$. Consequently, there exists a homomorphism
$\overline{\varphi}\colon B\to L$ extending $\varphi$. Then  $\ker
\overline{\varphi}$ is the desired ideal laying over $\frak p$.
\end{proof}

We shall give two proves of the following proposition.

\begin{proposition}\label{homst}
Let $A\subseteq B$ be difference rings, $B$ being difference
finitely generated over $A$, and there are only finitely many
minimal prime ideals in $A$. Then there exists a nonnilpotent
element $u$ in $A$ with the following property. For every difference
closed pseudofield $\Omega$ and every difference homomorphism
$\varphi\colon A \to \Omega$ such that $\varphi(u)\neq 0$ there
exists a difference homomorphism  $\overline{\varphi}\colon B\to
\Omega$ with the condition $\overline{\varphi}|_A=\varphi$.
\end{proposition}
\begin{proof}[First proof]

It follows from Theorem~\ref{equtheor} that $\Omega$ is of the
following form $\Fun L$, where $L$ is an algebraically closed field.
Let $\gamma_\sigma\colon \Omega\to L$ be the corresponding
substitution homomorphisms.

We shall reduce the theorem to the case where $A$ and $B$ are
reduced. Let us assume that we have proved the theorem for rings
without nilpotent elements. Let $\frak a$ and $\frak b$ be the
nilradicals of $A$ and $B$, respectively. Let $s'\in A/\frak a$ be
the desired element, denote by $s$ some preimage of $s'$ in $A$. Let
$\varphi\colon A\to \Omega$ be a difference homomorphism with
condition $\varphi(s)\neq0$. Since $\Omega$ does not contain
nilpotent elements, $\frak a$ is in the kernel of $\varphi$.
Consequently, there exists a homomorphism $\varphi'\colon A/\frak
a\to \Omega$.
$$
\xymatrix@R=10pt@C=15pt{
    &A\ar[r]\ar[d]\ar[dl]_{\varphi}&B\ar[d]\\
    \Omega&A/\frak a\ar[r]\ar[l]_{\varphi'}&B/\frak b
}
$$
Since $\varphi(s')=\varphi(s)\neq0$, then it follows from our
hypothesis that there is a difference homomorphism
$\overline{\varphi}'\colon B/\frak b\to \Omega$.
$$
\xymatrix@R=10pt@C=15pt{
    &A\ar[r]\ar[d]\ar[dl]_{\varphi}&B\ar[d]\\
    \Omega&A/\frak a\ar[r]\ar[l]_{\varphi'}&B/\frak
    b\ar@/^1pc/[ll]^{\overline{\varphi}'}
}
$$
Then the desired homomorphism $\overline{\varphi}\colon B\to \Omega$
is the composition of the quotient homomorphism and
$\overline{\varphi}'$.

Now we suppose that the nilradicals of $A$ and $B$ are zero. By
Lemma~\ref{lemma1}, it follows that there exists an element $s\in A$
such that $A_s$ contains only one minimal prime ideal. Since $A$ has
no nilpotent elements, $A_s$ is an integral domain. Let us apply
Lemma~\ref{lemma2} to the pair  $A_s\subseteq B_s$. So, there exists
an element $t\in A$ such that for any algebraically closed field $L$
every homomorphism of $A_{st}\to L$ can be extended to a
homomorphism $B_{st}\to L$. Denote the element $st$ by $u$. Let us
show that the desired property holds. Let $\varphi\colon A\to
\Omega$  be a difference homomorphism such that  $\varphi(u)\neq0$.
Then for some $\sigma$ we have  $\gamma_\sigma\circ
\varphi(u)\neq0$. So, there is a homomorphism  $\varphi_\sigma\colon
A\to L$ such that $\varphi_\sigma(u)\neq 0$. We shall extend
$\varphi_\sigma$ to a homomorphism $B\to L$ as it shown in the
following diagram (numbers show the order).
$$
\xymatrix@R=10pt@C=10pt{
    A_u\ar[rr]&&B_u&A_u\ar[rd]^{1}\ar[rr]&&B_u&A_u\ar[rd]^{1}\ar[rr]&&B_u\ar[dl]_{2}\\
    &L&\ar@{.>}[r]&&L&\ar@{.>}[r]&&L&&\\
    A\ar[uu]\ar[rr]^{\varphi}\ar[ru]^{\varphi_\sigma}&&\Omega\ar[ul]_{\gamma_\sigma}&A\ar[uu]\ar[rr]^{\varphi}\ar[ru]^{\varphi_\sigma}&&\Omega\ar[ul]_{\gamma_\sigma}&A\ar[uu]\ar[rr]^{\varphi}\ar[ru]^{\varphi_\sigma}&&\Omega\ar[ul]_{\gamma_\sigma}
}
$$

The homomorphism~$1$ appears from the condition
$\varphi_\sigma(u)\neq0$ and the universal property of localization.
The homomorphism~$2$ exists because of the definition of $u$.
$$
\xymatrix@R=10pt@C=10pt{
A_u\ar[rd]^{1}\ar[rr]&&B_u\ar[dl]_{2}&B\ar[l]\ar[dll]_{3}&A_u\ar[rd]^{1}\ar[rr]&&B_u\ar[dl]_{2}&B\ar[l]\ar[dll]_{3}\ar[ddl]_{\overline{\varphi}}\\
    &L&&\ar@{.>}[r]&&L&&\\
    A\ar[uu]\ar[rr]^{\varphi}\ar[ru]^{\varphi_\sigma}&&\Omega\ar[ul]_{\gamma_\sigma}&&A\ar[uu]\ar[rr]^{\varphi}\ar[ru]^{\varphi_\sigma}&&\Omega\ar[ul]_{\gamma_\sigma}&
}
$$
The homomorphism~$3$ is constructed as a composition. Then from
Proposition~\ref{taylor} there exists a difference homomorphism
$\overline{\varphi}$. Since the diagram is commutative, it follows
from the uniqueness of the Taylor homomorphism for $A$ that the
restriction of $\overline{\varphi}$ onto $A$ coincides with
$\varphi$.
\end{proof}
\begin{proof}[Second proof]
We shall derive this proposition from Proposition~\ref{inher}. Since
$B$ is finitely generated over $A$ then there exists an element $s$
in $A$ such that the corresponding mapping
$$
\Spec B_s\to \Spec A_s
$$
is surjective. Then it follows from Proposition~\ref{inher}~(1) that
the mapping
$$
(\PSpec B)_s\to (\PSpec A)_s
$$
is surjective.

Let $\Omega$ be an arbitrary difference closed pseudofield and
$\varphi\colon A\to \Omega$ is a difference homomorphism such that
$\varphi(s)\neq 0$. The kernel of $\varphi$ will be denoted by
$\frak p$ and we have $s\notin \frak p$. Therefore, there is a
pseudoprime ideal $\frak q\subseteq B$ such that $\frak q^c=\frak p$
and $s\notin \frak q$. Consider the following ring
$$
R=B/\frak q\mathop{\otimes}_{A/\frak p}\Omega.
$$
It follows from the definition that $R$ is difference finitely
generated over $\Omega$. Let $\frak m$  be an arbitrary maximal
difference ideal of $R$. Since $\Omega$ is difference closed, the
quotient ring $R/\frak m$ coincides with $\Omega$. Now we have the
following diagram
$$
\xymatrix@R=10pt{
    B\ar[rr]&&B/\frak q\ar[rd]&&\\
    A\ar[r]\ar[u]&A/\frak p\ar[ur]\ar[rd]&&B/\frak q\mathop{\otimes}_{A/\frak p}\Omega\ar[r]&R/\frak m = \Omega\\
    &&\Omega\ar[ru]&&\\
}
$$
The composition of upper arrows gives us the desired homomorphism
from $B$ to $\Omega$.
\end{proof}

There are two important particular cases of this proposition.

\begin{corollary}\label{cor1}
Let $A\subseteq B$  be difference rings, $B$ being difference
finitely generated over $A$, and $A$ is a pseudo integral domain.
Then there exists a non nilpotent element $u$ in $A$ with the
following property. For every difference closed pseudofield $\Omega$
and every difference homomorphism $\varphi\colon A \to \Omega$ such
that $\varphi(u)\neq 0$ there exists a difference homomorphism
$\overline{\varphi}\colon B\to \Omega$ with condition
$\overline{\varphi}|_A=\varphi$.
\end{corollary}
\begin{proof}

Since $A$ is a pseudo integral domain there are finitely many
minimal prime ideals in $A$. Indeed, let $\frak p$ be a minimal
prime ideal. Then it is $\Sigma$-associated with zero ideal. So,
$\cap_\sigma \frak p^\sigma=0$. Let $\frak q$ be an arbitrary
minimal prime ideal of $A$. Then  $\cap_\sigma\frak
p^\sigma=0\subseteq \frak q$. Therefore, for some $\sigma$ we have
$\frak p^\sigma\subseteq \frak q$. But $\frak q$ is a minimal prime
ideal, hence, $\frak p^\sigma=\frak q$. So, $\frak p^\sigma$ are all
minimal prime ideals of $A$. Now the result follows from the
previous theorem.
\end{proof}

\begin{corollary}\label{cor2}
Let $A\subseteq B$ be difference rings, $B$ being difference
finitely generated over $A$, and $A$ is a difference finitely
generated algebra over a pseudofield. Then there exists a non
nilpotent element $u$ in $A$ with the following property. For every
difference closed pseudofield $\Omega$ and every difference
homomorphism $\varphi\colon A \to \Omega$ such that $\varphi(u)\neq
0$ there exists a difference homomorphism  $\overline{\varphi}\colon
B\to \Omega$ with condition $\overline{\varphi}|_A=\varphi$.
\end{corollary}
\begin{proof}
Every pseudofield is an Artin ring and, thus, is Noetherian. If $A$
is difference finitely generated over a pseudofield, then $A$ is
finitely generated over it. Hence, $A$ is Noetherian. Consequently,
there are finitely many minimal prime ideals in $A$. Now the result
follows from the previous theorem.
\end{proof}

\begin{proposition}
Let $K$ be a pseudofield and $L$ be its difference closure. Consider
an arbitrary difference finitely generated algebra $A$ over $K$ and
a non nilpotent element $u\in A$. Then there is a difference
homomorphism $\varphi\colon A\to L$ such that $\varphi(u)\neq 0$.
\end{proposition}
\begin{proof}
As we know $L=\Fun (F)$ for some algebraically closed field $F$ and
there are homomorphisms $\gamma_\sigma\colon L\to F$. Then we have
the compositions $\pi_\sigma \colon K\to L\to F$. As we can see
every maximal ideal of $K$ can be presented as $\ker \pi_\sigma$ for
an appropriate element $\sigma$. So, every factor field of $K$ can
be embedded into $F$.

Since $u$ is not a nilpotent element then the algebra $A_u$ is not a
zero ring. Consider an arbitrary maximal ideal $\frak n$ in $A_u$.
Let $\frak m$ denote its contraction to $K$. Then $(A/\frak n)_u$ is
a finitely generated field over $K/\frak m$. The field $K/\frak m$
can be embedded into $F$ by some mapping $\pi_\sigma$. Since $F$ is
algebraically closed, there is a mapping $\phi_\sigma\colon (A/\frak
n)_u\to F$ such that the following diagram is commutative
$$
\xymatrix@R=15pt@C=15pt{
    (A/\frak n)_u\ar[rd]^{\phi_\sigma}&\\
    K/\frak m\ar[r]^{\pi_\sigma}\ar[u]&F
}
$$
By Proposition~\ref{taylor}, it follows that there exists a mapping
$\varphi\colon A\to L $ such that the following diagram is
commutative
$$
\xymatrix{
    (A/\frak n)_u\ar[rrd]^{\phi_\sigma}&A\ar[l]^{\pi}\ar[r]^{\varphi}&L\ar[d]^{\gamma_\sigma}\\
    &K/\frak m\ar[r]^{\pi_\sigma}\ar[ul]&F
}
$$
where $\pi\colon A\to (A/\frak n)_u$ is a natural mapping.  Since
$\phi_\sigma\circ\pi(u)\neq 0$, we have $\varphi(u)\neq 0$.
\end{proof}

The next technical condition is concerned with extensions of
pseudoprime ideals.

\begin{proposition}
Let $A\subseteq B$ be difference rings, $B$ being difference
finitely generated over $A$, and there are finitely many minimal
prime ideals in $A$. Then there exists an element $u$ in $A$ such
that the mapping
$$
(\PSpec B)_u\to(\PSpec A)_u
$$
is surjective.
\end{proposition}
\begin{proof}
We may suppose that the nilradicals of the rings are zero. By
Lemma~\ref{lemma1}, it follows that there exists an element  $s\in
A$ such that $A_s$ is an integral domain. Further, as in
Lemma~\ref{lemma2}  there is an element $t$ such that $B_{st}$  is
integral over $A_{st}[x_1,\ldots,x_n]$ and the elements
$x_1,\ldots,x_n$ are algebraically independent over $A_{st}$. Let
$u=st$. By Theorem~\cite[chapter~5, th.~5.10]{AM}, it follows that
the mapping $\Spec B_u\to \Spec A_u[x_1,\ldots,x_n]$ is surjective.
It is clear that the mapping
$$
\Spec A_u[x_1,\ldots,x_n]\to \Spec A_u
$$
is surjective too. So, from Proposition~\ref{inher} the mapping
$$
(\PSpec B)_u\to(\PSpec A)_u
$$
is surjective.
\end{proof}

\begin{proposition}\label{Pmaxst}
Let $A\subseteq B$ be difference finitely generated algebras over a
pseudofield. Then there exists an element $u$ in $A$ such that the
mapping
$$
(\PMax B)_u\to(\PMax A)_u
$$
is surjective.
\end{proposition}
\begin{proof}
Since the algebra $A$ is difference finitely generated over a
pseudofield $A$ is Noetherian. Consequently, there are finitely many
minimal prime ideals in $A$. Following the proof of the previous
proposition, we are finding the element $u$ such that the mapping
$\Spec B_u\to\Spec A_u$ is surjective. Since $A_u$ and $B_u$ are
finitely generated over an Artin ring, the contraction of any
maximal ideal is a maximal ideal. So, the mapping $\Max B_u\to\Max
A_u$ is well-defined and surjective. Then
Proposition~\ref{techbasic}~(2) completes the proof.
\end{proof}

\subsection{Geometry}\label{sec45}

In this section we develop a geometric theory of difference
equations with solutions in pseudofields. This theory is quite
similar to the theory of polynomial equations.

Let $A$ be a difference closed pseudofield. The ring of difference
polynomials $A\{y_1,\ldots,y_n\}$ will be denoted by $R_n$. For
every subset $E\subseteq R_n$ we shall define the subset $V(E)$ of
$A^n$ as follows
$$
V(E)=\{\,a\in A^n \mid \forall f\in E:\:f(a)=0\,\}.
$$
This set will be called a pseudovariety. Conversely, let $X$ be an
arbitrary subset in $A^n$, then we set
$$
I(X)=\{\,f\in R_n\mid f|_X=0\,\}.
$$
This ideal is called the ideal of definition of $X$. Let now
$\Hom^\Sigma_A(R_n,A)$ denote the set of all difference
homomorphisms from $R_n$ to $A$ over $A$. Consider the mapping
$$
\varphi\colon A^n\to \Hom^\Sigma_A(R_n,A)
$$
by the rule: every point $a=(a_1,\ldots,a_n)$  maps to a
homomorphism $\xi_a$ such that $\xi_a(f)=f(a)$. The mapping
$$
\psi\colon\Hom^\Sigma_A(R_n,A)\to\PMax A
$$
by the rule $\xi\mapsto\ker\xi$ will be denoted by $\psi$. So, we
have the following sequence
$$
A^n\stackrel{\varphi}{\longrightarrow}
\Hom^\Sigma_A(R_n,A)\stackrel{\psi}{\longrightarrow}\PMax A.
$$

\begin{proposition}\label{bij}
The mappings $\varphi$ and $\psi$ are bijections
\end{proposition}
\begin{proof}
The inverse mapping for $\varphi$ is given by the rule
$$
\xi\mapsto(\xi(y_1),\ldots,\xi(y_n)).
$$
Since $A$ is difference closed, for every homomorphism $\xi\colon
R_n\to A$ its kernel is of the form $\ker
\xi=[y_1-a_1,\ldots,y_n-a_n]$, where $a_i=\xi(y_i)$. So, the mapping
$\psi$ injective and surjective.
\end{proof}

It is clear that under the mapping $\psi\circ\varphi$ the set $V(E)$
of $A^n$ maps to the set $V(E)$ of $\PMax R_n$. So, the sets $V(E)$
define a topology on $A^n$ and the mentioned mapping is a
homeomorphism. Therefore, we can identify pseudomaximal spectrum of
$R_n$ with an affine space $A^n$. Let $\frak a$ be a difference
ideal in $R_n$. Then the set $\Hom^\Sigma_A(R_n/\frak a,A)$ can be
identified with the set of all homomorphisms of
$\Hom^\Sigma_A(R_n,A)$ mapping $\frak a$ to zero. In other words,
there is a homeomorphism between $V(\frak a)$ and $\PMax R_n/\frak
a$.

\begin{corollary}
The mappings $\varphi$ and $\psi$ are homeomorphisms.
\end{corollary}

\begin{theorem}\label{nullth}
Let $\frak a$ be a difference ideal in $R_n$. Then $\frak r(\frak
a)=I(V(\frak a))$.
\end{theorem}
\begin{proof}
Since $A$ is an Artin ring, $A$ is a Jacobson ring. $R_n$ is
finitely generated over $A$, consequently, $R_n$ is a Jacobson ring
too. Therefore, every radical ideal in $R_n$ can be presented as an
intersection of maximal ideals. Hence, every radical difference
ideal can be presented as an intersection of pseudomaximal ideals
(Proposition~\ref{techbasic} item~(2)). Now we are useing the
correspondence between points of $V(\frak a)$ and pseudomaximal
ideals (Proposition~\ref{bij}).
\end{proof}

\subsection{Regular functions and a structure sheaf}\label{sec46}

Let $X\subseteq A^n$ be a pseudovariety over a difference closed
pseudofield $A$ and let $I(X)$ be its ideal of definition in the
ring $R_n=A\{y_1,\ldots,y_n\}$. Then the ring $R_n/I(X)$ can be
identified with the ring of polynomial functions on $X$ and will be
denoted by $A\{X\}$.

Let $f\colon X\to A$ be a function. We shall say that $f$ is regular
at $x\in X$ if there are an open neighborhood $U$ containing $x$ and
elements $h,g\in A\{x_1,\ldots,x_n\}$ such that for every element
$y\in U$ $g(y)$ is invertible and $f(y)=h(y)/g(y)$. The condition on
$g$ can be stated as follows: for each element $y\in U$ the value
$g(y)$ is not a zero divisor. For any subset $Y$ in $X$ a function
is said to be regular on $Y$ if it is regular at each point of $Y$.
The set of all regular functions on an open subset $U$ of $X$ will
be denoted by $\mathcal O_X(U)$. Since the definition of a regular
function arises from a local condition the set of the rings
$\mathcal O_X$ form a sheaf. This sheaf will be called a structure
sheaf on $X$. This definition naturally generalizes the usual one in
algebraic geometry. It follows from the definition that there is the
inclusion $A\{X\}\subseteq \mathcal O_X(X)$. The very important fact
is that the other inclusion is also true.

\begin{theorem}\label{regularf}
For an arbitrary pseudovariety $X$ there is the equality
$$
A\{X\}=\mathcal O_X(X).
$$
\end{theorem}
\begin{proof}
Let $f$ be a regular function on $X$. It follows from the
definitions of a regular function that for each point $x\in X$ there
exist a neighborhood $U_x$ and elements $h_x,g_x\in A\{X\}$ such
that for every element $y\in U_x$ $g_x(y)$ is invertible and
$$
f(y)=h_x(y)/g_x(y).
$$
Replacing $h_x$ and $g_x$ by $h_xg_x$ and $g_x^2$, respectively, we
can suppose that the condition $g_x(y)=0$ implies $h_x(y)=0$. The
element $\prod_{\sigma}\sigma(g_x)$ is $\Sigma$-constant. So, we
replace $h_x$ and $g_x$ by
$$
h_x\prod_{\sigma\neq e}\sigma(g_x)\quad \mbox{ and
}\quad\prod_{\sigma}\sigma(g_x).
$$ Hence, we can suppose that each
$g_x$ is $\Sigma$-constant.

The family $\{U_x\}$ covers $X$ and $X$ is compact. So, there is a
finite subfamily such that
$$
X=U_{x_1}\cup\ldots\cup U_{x_m}.
$$
For every $\Sigma$-constant element $s$ the set $X_s$ coincides with
the set of all points where $s$ is invertible. Therefore, we have
$U_x\subseteq X_{g_x}$. Hence, $X_{g_{x_i}}$ cover $X$ and, thus,
$(g_{x_1},\ldots,g_{x_m})=(1)$. So, we have
$$
1=d_1 g_{x_1}+\ldots+d_m g_{x_m}.
$$

Now observe that $h_{x}g_{x'}=h_{x'}g_x$, where $x$ and $x'$ are
among $x_i$. Indeed, if
$$
g_{x}(y)=0\quad \mbox{ or }\quad
g_{x'}(y)=0
$$
then the other part of the equality is zero. If both $g_{x}$ and
$g_{x'}$ are not zero then the condition holds because they define
the same functions on intersection $X_{g_{x}}\cap X_{g_{x'}}$. Now
set $d=\sum_i d_i h_{x_i}$. We claim that  $f=d$. Indeed,
$$
dg_{x_j}=\sum d_i h_{x_i}g_{x_j}=\sum d_i h_{x_j}g_{x_i}=h_{xj}.
$$
\end{proof}

The given definition of a structure sheaf comes from algebraic
geometry. Roughly speaking, we use inverse operation to produce a
rational function. When we deal with a field we know that inverse
operation is defined for every nonzero element. In the case of
pseudofields we have a similar operation. Indeed, for every nonzero
element $a$ of an arbitrary absolutely flat ring there exist unique
elements $e$ and $a^*$ with relations
$$
a=ea,\:\:a^*=ea^*,\:\: e = a a^*.
$$
A formal definition of these elements is the following. For every
element $a$ of an absolutely flat ring $A$ there is an element $x\in
A$ such that $a=xa^2$. Then set $e=ax$ and $a^*=ax^2$. These
elements can be described as follows. The element $a$ can be
considered as a function on the spectrum of $A$. Then the element
$e$ can be defined as a function that equals $1$ where $a$ is not
zero and equals zero, where $a$ is zero. In other words, $e$ is the
indicator of the support of $a$. So, the element $a^*$ is equal to
the inverse element where $a$ is not zero and zero otherwise.

Now we shall define the second structure sheaf. But we will see that
this new sheaf coincides with the sheaf above. Let $X\subseteq A^n$
be a pseudovariety over a difference closed pseudofield $A$.
Consider an arbitrary function $f\colon X\to A$. We shall say that
$f$ is pseudoregular at a given point $x\in X$ if there exist a
neighborhood $U$ containing $x$ and elements $h,g\in
A\{x_1,\ldots,x_n\}$ such that for every $y\in U$ the element $g(y)$
is not zero and $f(y)=h(y)(g(y))^*$. The function pseudoregular at
each point of the subset $Y$ is called pseudoregular on $Y$. The set
of all pseudoregular functions on an open set $U$ will be denoted by
$\mathcal O'_X$. Let us note that there is a natural mapping
$\mathcal O_X\to \mathcal O'_X$. Actually, the sheaf $\mathcal O_X$
is a subsheaf of $\mathcal O'_X$. Let us show that both sheaves
coincide to each other.

\begin{theorem}\label{regularfnew}
Under the above assumptions, we have
$$
\mathcal O_X=\mathcal  O'_X.
$$
\end{theorem}
\begin{proof}
Let $f$ be a pseudoregular function defined on some open subset of
$X$ and let $x$ be an arbitrary point, where $f$ is defined. Then it
follows from the definition of pseudoregular function that there are
neighborhood $U$ of $x$ and elements $h,g\in A\{X\}$ such that for
every point $y\in U$ $g(y)$ is not zero and
$$
f(y)=h(y)(g(y))^*.
$$
Let $e$ be an idempotent of $A$ corresponding to $g(x)$. Then in
some smaller neighborhood (we should intersect $U$ with $X_{eg}$)
$f$ is given by the equality
$$
f(y)=eh(y)(eg(y))^*.
$$
Let us set $g'(y)= 1-e + eg(y)$. So, the value $g'(x)$ is invertible
in $A$. Now we consider the functions
$$
h_0 = eh\prod_{\sigma\neq e} \sigma(g')\quad \mbox{ and } \quad g_0
= \prod_{\sigma}\sigma(g').
$$
So, $g_0$ is a $\Sigma$-constant element. Therefore, $g_0(y)$ is
invertible for every $y\in X_{g_0}$. The set $X_{g_0}$ contains $x$
because $g'(x)$ is invertible. Additionally, we have
$$
eh(y)(eg(y))^*=h_0(y)/g_0(y)
$$
for all $y\in U\cap X_{g_0}$. Therefore, if a function $f$ is
pseudoregular at $x$ it is also regular at $x$.
\end{proof}

So, as we can see the new method of constructing structure sheaf
gives us the same result.

Let $X\subseteq A^n$ and $Y\subseteq A^m$ be pseudovarieties over a
difference closed pseudofield $A$. A mapping $f\colon X\to Y$ will
be called regular if its coordinate functions are regular. By
Theorem~\ref{regularf}, it follows that the set of all regular
functions on $X$ coincides with the set of polynomial functions and,
therefore, every regular mapping from $X$ to $Y$ coincides with a
polynomial mapping.

Let a mapping $f\colon X\to Y$  be a regular one. So, all functions
$f_i(x_1,\ldots,x_n)$ are difference polynomials. For every such
$f\colon X\to Y$ there is a difference homomorphism $f^*\colon
A\{Y\}\to A\{X\}$ by the rule $f^*(\xi)=\xi\circ f$.

Conversely, for every difference homomorphism  $\varphi\colon
A\{Y\}\to A\{X\}$ over $A$ we shall define
$$
\varphi^*\colon \Hom^\Sigma_A(A\{X\},A)\to \Hom^\Sigma_A(A\{Y\},A)
$$
by the rule $\varphi^*(\xi)=\xi\circ\varphi$. Let us recall that the
pseudovariety $X$ can be identified with $\Hom^\Sigma_A(A\{X\},A)$.
Then we have the mapping
$$
\varphi^*\colon X\to Y.
$$

\begin{proposition}
The constructed mappings are inverse to each other bijections
between the set of all regular mappings from $X$ to $Y$ and the set
of all difference homomorphisms from $A\{Y\}$ to $A\{X\}$.
\end{proposition}
\begin{proof}

If $f\colon A^n\to A^m$ is a polynomial mapping then $f(X)\subseteq
Y$ iff $f^*(I(Y))\subseteq I(X)$. Since
$$A\{X\}=A\{y_1,\ldots,y_n\}/I(X)$$ and
$$A\{Y\}=A\{y_1,\ldots,y_m\}/I(Y)$$ then the set of all
$$
g\colon A\{y_1,\ldots,y_m\}\to A\{y_1,\ldots,y_n\}
$$
with condition $g(I(Y))\subseteq I(X)$ corresponds to the set of all
$\bar g\colon A\{Y\}\to A\{X\}$.
\end{proof}

\subsection{Geometry continuation}\label{sec47}

Here we shall continue investigation of some geometric properties of
pseudovarieties and their morphisms.

Since every pseudofield is an Artin ring and a difference finitely
generated algebra over a pseudofield is finitely generated, then
every algebra difference finitely generated over a pseudofield is
Noetherian. So, we have the following.

\begin{proposition}\label{netst}
Every pseudovariety is a Noetherian topological space.
\end{proposition}

The following propositions are devoted to the geometric properties
of regular mappings.

\begin{proposition}\label{imst}
Let $f\colon X\to Y$ be a regular mapping with the dense image and
let $Y$ be irreducible. Then the image of $f$ contains an open
subset.
\end{proposition}
\begin{proof}
Let $A\{X\}$ and $A\{Y\}$ be coordinate rings of the
pseudovarieties. Then the mapping $f$ gives us the mapping
$$
f^*\colon A\{Y\}\to A\{X\}.
$$
Since the image of $f$ is dense, the homomorphism $f^*$ is
injective. By Proposition~\ref{Pmaxst}, it follows that there exists
an element $s\in A\{Y\}$ such that the mapping  $\PMax A\{X\}_s\to
\PMax A\{Y\}_s$ is surjective. But from Proposition~\ref{bij} the
last mapping coincides with  $f\colon X_s\to Y_s$. Since $Y$ is
irreducible, every open subset is dense.
\end{proof}

\begin{proposition}\label{constrst}
Let $f\colon X\to Y$ be a regular mapping. Then $f$ is
constructible.
\end{proposition}
\begin{proof}
Let $A\{X\}$ and $A\{Y\}$ be denoted by $B$ and $D$, respectively.
Then we have the corresponding difference homomorphism $f^*\colon
D\to B$. We identify pseudovarieties $X$ and $Y$ with pseudomaximal
spectra of the rings $B$ and $D$, respectively.

Let $E$ be a constructible subset in $X$, then it has the form
$$
E=U_1\cap V_1\cup\ldots\cup U_n\cap V_n,
$$
where $U_i$ are open and $V_i$ are closed. Since the image of
mapping preserves a union of sets then we can suppose that $E=U\cap
V$, where $U$ is open and $V$ is irreducible and closed. Let $V$ be
of the form $V=V(\frak p)$, where $\frak p$ is a pseudoprime ideal
of $B$. Taking a quotient by $\frak p$ we reduce to the case where
$E$ is open in $X$ and $X$ is irreducible.

Now we are going to show that $f(E)$ is constructible in $Y$. To do
this we shall use a criterion~\cite[chapter~7, ex.~21]{AM}. Let
$X_0$ be an irreducible subset in $Y$ such that $f(E)$ is dense in
$X_0$. Here we must show that the image of $E$ contains an open
subset in $X_0$. Then $X_0$ is of the form $V(\frak p)$, where
$\frak p$ is a pseudoprime ideal in $D$. The preimage of $X_0$ under
$f$ is of the form $V(\frak p^e)$, where $\frak p^e$ is the
extension of $\frak p$ to $B$. The closed set $V(\frak p^e)$ can be
presented as follows
$$
V(\frak p^e)=V(\frak q_1)\cup\ldots\cup V(\frak q_m),
$$
where $\frak q_i$ are pseudoprime ideals of $B$. Therefore, the set
$E$ has the following decomposition
$$
E=U\cap V(\frak q_1)\cup\ldots\cup U\cap V(\frak q_m).
$$
Considering quotient by $\frak p$ and $\frak p^e$ we reduce the
problem to the case $D$ is a pseudodomain and $Y=X_0$. Now the image
of $E$ is of the form
$$
f(E)=f(V(\frak q_1)\cap U)\cup\ldots\cup f(V(\frak q_m)\cap U).
$$
Since $f(E)$ is dense in $Y$ and $Y$ is irreducible then there
exists an $i$ such that $f(V(\frak q_i)\cap U)$ is dense in $Y$.
Replacing $B$ by $B/\frak q_i$ we can suppose that $D\subseteq B$,
$B$ is pseudodomain and $E=U$ is open. Every open subset is a union
of principal open subsets, and since $X$ is a Noetherian topological
space this union is finite. Therefore, we can suppose that $E=X_s$
and we need to prove that $f(E)$ contains open subset in $Y$.

In order to show the last claim we shall prove that there is a
nonzero element $t\in D$ such that the mapping $\Max B_{st}\to \Max
D_t$ is surjective. Then our proposition follows from
Proposition~\ref{techbasic}~(2). First of all we note that every
minimal prime ideal of $B$ is $\Sigma$-associated with the zero
ideal. Since $D$ is a subring of $B$ then the contraction of every
$\Sigma$-associated with zero prime ideal is $\Sigma$-associated
with zero prime ideal in $D$. So, every minimal prime ideal of $B$
contracts to a minimal prime ideal of $D$. Now there exists a
minimal prime ideal $\frak q$ of $B$ such that $s\notin\frak q$. The
contraction of $\frak q$ to $D$ will be denoted by $\frak p$. Let
$\frak p_1,\ldots,\frak p_n$ be the set of all minimal prime ideals
of $D$, where $\frak p=\frak p_1$. Then there is an element $t\in D$
such that $t\in \bigcup_{i=2}^n\frak p_i\setminus \frak p_1$. So, we
have the inclusion $D_t\subseteq B_t$. Moreover, the element $t$ was
constructed such that $D_t=(D/\frak p)_t$. Therefore, the
composition of embedding $D_t\to B_t$ and localization $B_t\to
B_{st}$ is injective. So, $B_{st}$ is a finitely generated algebra
over an integral domain $D_t$. Therefore, there exists an element
$u\in D_t$ such that the mapping $\Max B_{stu}\to \Max D_{tu}$ is
surjective. Denoting the element $tu$ by $t$ we complete the proof.
\end{proof}

\begin{proposition}\label{openst}
Let $f\colon Y\to X$ be a regular mapping with the dense image. Then
there exists an element $u\in A\{X\}$ such that the mapping $f\colon
Y_{f^*(u)}\to X_u$  is open.
\end{proposition}
\begin{proof}

Let the rings $A\{X\}$ and $A\{Y\}$ be denoted by $C$ and $D$,
respectively. Since the image of $f$ is dense, $f^*\colon C\to D$ is
injective. So, $C$ can be identified with the subring in $D$. By
Lemma~\ref{lemma1}, it follows that there exists an element $s\in C$
such that $C_s$ is an integral domain and $C_s\subseteq D_s$. Since
$C_s$ is an integral domain and $D_s$ is finitely generated over
$C_s$, there exists an element  $t\in C$ such that $D_{st}$ is a
free $C_{st}$-module (see~\cite[chapter~8, sec.~22, th.~52]{Mu}).
Let us denote $st$ by $u$. Then $D_u$ is a faithfully flat algebra
over $C_u$, thus, by~\cite[chapter~3, ex.~16]{AM}
and~\cite[chapter~5, ex.~11]{AM} the corresponding mapping $\Spec
D_u\to \Spec C_u$ has the going-down property and surjective.
Proposition~\ref{inher}~(1) and~(3) guaranties that the mapping
$(\PSpec D)_u\to (\PSpec C)_u$ is surjective and has the going-down
property. Let us show that $f\colon Y_u\to X_u$ is open.

It suffices to show that the image of every principal open set is
open. Let $Y_t$ be a principal open set, then
$$
Y_u\cap Y_t=\cup_{\sigma \tau} Y_{\sigma(u)\tau(t)}.
$$
It suffices to consider the set of the form $Y_{\sigma(u)\tau(t)}$.
Since $Y_{w}=Y_{\sigma(w)}$, it suffices to consider the set of the
form $Y_{uv}$.

To show that the set $f(Y_{uv})$ is open we shall use the
criterion~\cite[chapter~7, ex.~22]{AM}. Let $Y'$ and $X'$ be
pseudospectra of $D$ and $C$, respectively. Not that every
irreducible closed subset in $X$ has the following form  $X'_0\cap
X_u$, where $X'_0$ is an irreducible subset in $X'$. Consider the
set $f(Y'_{uv})$ and let $X'_0$ be an irreducible closed subset in
$X'$. Consider $f(Y'_{uv})\cap X'_0$. Suppose that the last set is
not empty. We have
$$
f(Y'_{uv})\cap X'_0=f(Y'_{uv}\cap f^{-1}(X'_0)).
$$
Let $X'_0=V(\frak q)$, where $\frak q\in \PSpec C$. Therefore,
$$
f(Y'_{uv})\cap X'_0=f(Y'_{uv}\cap V(\frak q^e)).
$$
The last set is not empty. Thus, there exists a prime ideal $\frak
q'$ in $D$ such that $\frak q^e\subseteq\frak q'$ and $uv\notin
\frak q'$. Since  $D_{uv}$ is a flat $C_u$-module, using the same
arguments as above, we see that the mapping $\Spec D_{uv}\to \Spec
C_{u}$ has the going-down property. Therefore, the mapping $f\colon
Y'_{uv}\to X_u$ has the going-down property. Now consider the chain
of pseudoprime ideals $(\frak q')^c\supseteq \frak q$ in $C$ and
$\frak q'$ in $D$. Then there exists a pseudoprime ideal $\frak q''$
in $D$ such that $(\frak q'')^c=\frak q$. Therefore, homomorphism
$C/\frak q\to D/\frak q^e$ is injective. Now consider the pair of
rings
$$
(C/\frak q)_u\subseteq(D/\frak q^e)_{uv}.
$$
By Lemma~\ref{lemma1}, it follows that there exists an element $s\in
C/\frak q$ such that $(C/\frak q)_{us}$ is an integral domain. Then
Lemma~\ref{lemma3} guaranties that for some element $t\in (C/\frak
q)_{su}$ the mapping
$$
\Spec (D/\frak q^e)_{uvst}\to \Spec (C/\frak q)_{ust}
$$
is surjective. Since the rings in the last expression are finitely
generated algebras over an Artin ring, the mapping
$$
\Max (D/\frak q^e)_{uvst}\to \Max (C/\frak q)_{ust}
$$
is surjective. By Proposition~\ref{techbasic}~(2), it follows that
the mapping
$$
(\PMax D/\frak q^e)_{uvst}\to (\PMax C/\frak q)_{ust}
$$
is surjective. Thus, $X_{ust}\cap (X'_0\cap X)$ is contained in
$f(Y_{uv})$. Now we are able to apply the criterion~\cite[chapter~7,
ex.~22]{AM}. To complete the proof we need to remember that $\PMax
C$ can be identified with  $X$ and $\PMax D$ with $Y$.
\end{proof}

\subsection{Adjoint construction}\label{sec48}

Let $M$ be an abelian group. Consider the set of all functions on
$\Sigma$ taking values in $M$. This set has a natural structure of
abelian group. We shall denote it by $\Fun(M)$. So, we have
$$
\Fun(M)=M^\Sigma=\{\,f\colon \Sigma\to M\,\}
$$
The group $\Sigma$ is acting on $\Fun(M)$ by the following rule
$(\tau f)(\sigma)=f(\tau^{-1}\sigma)$. Let $\sigma$ be an arbitrary
element of $\Sigma$, then we have a homomorphism of abelian groups
$\gamma_\sigma\colon \Fun(M)\to M$ by the rule
$\gamma_\sigma(f)=f(\sigma)$. For an arbitrary homomorphism of
abelian groups $h\colon M\to M'$ we have a homomorphism
$\Fun(h)\colon \Fun(M)\to \Fun(M')$ given by the rule $\Fun(h)f=hf$,
where $f\in\Fun(M)$. This homomorphism commutes with the action of
$\Sigma$. This group has the following universal property.

\begin{lemma}\label{taylormod}
Let $N$ and $M$ be abelian groups. Suppose that $\Sigma$ acts on $N$
by automorphisms. Then for every homomorphism of abelian groups
$\varphi\colon N\to M$ and every element $\sigma\in \Sigma$ there is
a unique homomorphism of abelian groups $\Phi_{\sigma}\colon N\to
\Fun(M)$ such that the diagram
$$
\xymatrix{
    &\Fun(M)\ar[d]^{\gamma_{\sigma}}\\
    N\ar[r]^{\varphi}\ar[ru]^{\Phi_\sigma}&M\\
}
$$
is commutative and $\Phi_\sigma$ commutes with the action of
$\Sigma$.
\end{lemma}
\begin{proof}
If such a homomorphism $\Phi_\sigma$ exists then it satisfies the
property $\nu \Phi_\sigma(n)=\Phi_\sigma (\nu n)$, where
$\nu\in\Sigma$ and $n\in N$. Therefore, we have
$$
\Phi_\sigma(n)(\tau)=\Phi_\sigma(n)((\tau\sigma^{-1})\sigma)=
\Phi_\sigma((\tau\sigma^{-1})^{-1}n)=\varphi(\sigma\tau^{-1}n)
$$
So, such homomorphism is unique.

For existence let us define $\Phi_\sigma$ by the rule
$$
\Phi_\sigma(n)(\tau)=\varphi(\sigma\tau^{-1}n).
$$
It is clear that this mapping is a homomorphism of abelian groups.
Now we shall check that it commutes with the action of $\Sigma$. Let
$\nu\in\Sigma$, then
$$
(\nu\Phi_{\sigma}(n))(\tau)=\Phi_{\sigma}(\nu^{-1}\tau)=\varphi(\sigma(\nu^{-1}\tau)^{-1}n)=
\varphi(\sigma\tau^{-1}\nu n)=\Phi_{\sigma}(\nu n)(\tau)
$$

\end{proof}

Let $A$ be a commutative ring with an identity. Let us denote the
category of all $A$ modules by $A{-}\mathbf{mod}$. For a given ring
$A$ we construct the ring $\Fun(A)$ and will denote the category of
all difference $\Fun(A)$ modules by
$\Sigma{-}\Fun(A){-}\mathbf{mod}$.

Let $M$ be an $A$ module, then we can produce an abelian group
$\Fun(M)$. The group $\Fun(M)$ has a structure of $\Fun(A)$ module
by the rule $(fh)(\sigma)=f(\sigma)h(\sigma)$, where $f\in\Fun(A)$
and $h\in \Fun(M)$. As we can see $\Fun(M)$ is a difference
$\Fun(A)$-module.

Let $f\colon M\to M'$ be a homomorphism of $A$-modules. Then the
homomorphism $\Fun(f)\colon \Fun(M)\to \Fun(M')$ is a difference
homomorphism of $\Fun(A)$ modules. As we can see these data define a
functor.

Let $N$ be a difference $\Fun(A)$ module. Let $e\in\Fun(A)$ be the
indicator of the identity element of $\Sigma$. Consider an abelian
group $eN$ in $N$. We have a homomorphism $A\to \Fun(A)$ such that
every element of $A$ maps to a constant function. Therefore, module
$N$ has a structure of $A$ module. Moreover, $eN$ is a submodule
under defined action of $A$. There is another way to provide an
action of $A$ on $eN$. We have a homomorphism
$\gamma_e\colon\Fun(A)\to A$. Then for any element $a\in A$ we can
take its preimage $x$ in $\Fun(A)$ and we define $an=xn$ for $n\in
eN$. This definition is well-defined. Indeed, let $x'$ be another
preimage of $a$, then $x-x'=(1-e)y$. Therefore, $xn=x'n$ for all
$n\in eN$. Particulary, for every $a\in A$ we can take a constant
function with a value $a$. Therefore, both defined actions coincide
to each other.

The second construction can be described in other terms. For the
homomorphism $\gamma_e\colon \Fun(A)\to A$ and a difference
$\Fun(A)$-module $N$ we consider a module
$N\mathop{\otimes}_{\Fun(A)} A_{\gamma_e}$, where index $\gamma_e$
reminds us the structure of $\Fun(A)$-module on $A$. So, we have a
functor in other direction.

\begin{theorem}\label{equivmod}
The functors
\begin{align*}
\Fun\colon& A{-}\mathbf{mod}\to
\Sigma{-}\Fun(A){-}\mathbf{mod}\\
{-}\mathop{\otimes}_{\Fun(A)}A_{\gamma_e}\colon&
\Sigma{-}\Fun(A){-}\mathbf{mod}\to A{-}\mathbf{mod}\\
\end{align*}
are inverse to each other equivalences.
\end{theorem}
\begin{proof}
Consider the composition $G\circ \Fun$. Let $M$ be an arbitrary
$A$-module, then $G\Fun (M)=e\Fun(M)$. Now we see that the
restriction of $\gamma_e$ on $e\Fun(M)$ is the desired natural
isomorphism.

Now, let $N$ be a difference $\Fun(A)$-module. Then we have $\Fun  G
(N)=\Fun(eN)$. Using Lemma~\ref{taylormod} we can define a
homomorphism $\Phi_N\colon N\to \Fun(eN)$. Let us show that this
homomorphism is a homomorphism of $\Fun(A)$-modules. Consider
$a\in\Fun (A)$ and $n\in N$. Then
$$
\Phi_N(an)(\tau)=e(\tau^{-1}(an))=e\tau^{-1}(a)\tau^{-1}(n)=ea_{\tau}\tau^{-1}(n)
$$
and
$$
(a\Phi_N(n))(\tau)=a_\tau \Phi_N(n)(\tau)=e a_\tau \tau^{-1}(n).
$$

We are going to show that $\Phi_N$ is a desired natural isomorphism.
For that we need to show that all $\Phi_N$ are isomorphisms. Let
$n\in N$. We have an equality $1=\sum_\sigma e_\sigma$, where
$e_\sigma$ is the indicator of the element $\sigma$. Then
$n=\sum_\sigma e_\sigma n$. Therefore, there is $\sigma$ such that
$e_\sigma n\neq 0$. Thus,
$e\sigma^{-1}n=\sigma^{-1}(e_{\sigma}n)\neq 0$. Thus, $\Phi_N(n)$ is
not a zero function.

Now consider the function $\Phi_N(e_\sigma \sigma(n))$. Let us
calculate its values.
$$
\Phi_N(e_\sigma \sigma(n))(\tau)=e(\tau^{-1}(e_\sigma\sigma(n)))=e
e_{\tau^{-1}\sigma}\tau^{-1}\sigma (n)= \left\{
\begin{aligned}
&en, &\tau=\sigma\\
&0, &\tau\neq\sigma\\
\end{aligned}
\right.
$$
Therefore, $\Phi_N(e_{\sigma}\sigma(n))=(en)e_{\sigma}$. Since every
element of $\Fun(eN)$ is of the form $\sum_\sigma
(en_\sigma)e_\sigma$ we get the desired.
\end{proof}

Now let $B$ be an $A$-algebra, then $\Fun(B)$ is a difference
$\Fun(A)$-algebra, and conversely if $D$ is an $\Fun(A)$-difference
algebra, then $D\mathop{\otimes}_{\Fun(A)}A_{\gamma_e}$ is an
$A$-algebra. Moreover, we see that for every ring homomorphism
$f\colon B\to B'$ the mapping $\Fun(f)\colon \Fun(B)\to\Fun(B')$ is
a difference homomorphism, and for arbitrary difference homomorphism
$h\colon D\to D'$ the mapping $h\mathop{\otimes} Id\colon
D\mathop{\otimes}_{\Fun(A)}A_{\gamma_e}\to
D'\mathop{\otimes}_{\Fun(A)}A_{\gamma_e}$ is a ring homomorphism.

Let us denote by $A{-}\mathbf{alg}$ the category of $A$-algebras and
by $\Sigma{-}\Fun(A){-}\mathbf{alg}$ the category of difference $\Fun(A)$-algebras.
So, we have proved a theorem.

\begin{theorem}\label{equivalg}
Functors
\begin{align*}
\Fun\colon& A{-}\mathbf{alg}\to
\Sigma{-}\Fun(A){-}\mathbf{alg}\\
{-}\mathop{\otimes}_{\Fun(A)}A_{\gamma_e}\colon&
\Sigma{-}\Fun(A){-}\mathbf{alg}\to A{-}\mathbf{alg}\\
\end{align*}
are inverse to each other equivalences.
\end{theorem}

For the difference ring $\Fun(A)$ there is a homomorphism
$\gamma_e\colon \Fun(A)\to A$. For every difference
$\Fun(A)$-algebra $B$ the last homomorphism induce a homomorphism
$B\to B\mathop{\otimes}_{\Fun(A)}A_{\gamma_e}$. We shall similarly
denote the last homomorphism  by $\gamma_e$.

\subsection{Adjoint variety}\label{sec49}

Let $X\subseteq A^n$ be a pseudovariety over a difference closed
pseudofield $A$. We know that $A=\Fun(K)$, where $K$ is an
algebraically closed field. We shall connect with $X$ a
corresponding algebraic variety over $K$. Moreover, if $X$ has a
group structure such that all group laws are regular mappings, then
the corresponding algebraic variety will be an algebraic group.

From the previous section we have the equivalence of categories
$K{-}\mathbf{alg}$ and $\Sigma{-}A{-}\mathbf{alg}$. Now we note that
a difference $A$-algebra $B$ is difference finitely generated over
$A$ iff the algebra $G(B)$ is finitely generated over $K$. For any
pseudovariety $X$ we construct the ring of regular functions
$A\{X\}$. This ring is difference finitely generated over $A$,
therefore, the ring $B=G(A\{X\})$ is finitely generated over $K$.
The last ring defines an algebraic variety $X^*$ such that the ring
of regular functions $K[X^*]$ coincides with $B$. The variety $X^*$
will be called an adjoint variety for $X$.

Now consider an arbitrary pseudovariety $X\subseteq A^n$. Then a
difference line $A$ has a natural structure of an affine space over
$K$. Indeed, $A=\Fun(K)=K^m$, where $m=|\Sigma|$. Therefore, the set
$X$ can be considered as a subset in $K^{mn}$. We claim that this
subset is an algebraic variety over $K$ that can be naturally
identified with $X^*$.

Now we shall show that there is a natural bijection between $X$ and
$X^*$. Indeed, pseudovariety $X$ can be naturally identified with
$$
\hom_{\Sigma{-}A{-}\mathbf{alg}}(A\{X\},A).
$$
From Theorem~\ref{equivalg} the last set can be naturally identified
with
$$
\hom_{K{-}\mathbf{alg}}(K[X^*],K).
$$
And the last set coincides with $X^*$. So, we have constructed a
mapping $\varphi\colon X\to X^*$.

We shall describe the bijection $\varphi$ explicitly. Consider a
pseudovariety $X$ over difference closed pseudofield $A$ and let
$A=\Fun(K)$, where $K$ is algebraically closed. Suppose that $X$ is
a subset of $A^n$. Then the algebra $A\{X\}$ is of the form
$A\{x_1,\ldots,x_n\}/I(X)$, where $I(X)$ is the ideal of definition
for $X$. For every point $(a_1,\ldots,a_n)\in X$ we construct a
difference homomorphism $A\{X\}\to A$ by the rule $f\mapsto
f(a_1,\ldots,a_n)$. Using this rule the variety $X$ can be
identified with
$$
\{\, \varphi\colon A\{X\}\to A \mid \varphi\mbox{ is a }\Sigma\mbox{
homomorphism },\: \varphi|_{A}=Id\,\}
$$
Then it follows from Theorem~\ref{taylor} that this set coincides
with
$$
\{\, \varphi\colon A\{X\}\to K\mid \varphi\mbox{ is a homomorphism},
\: \varphi|_{A}=\gamma_e \,\}
$$
Let $e\in A$ be the indicator of the identity element of $\Sigma$.
Then for every homomorphism $\varphi$ such that
$\varphi|_{A}=\gamma_e$ the element $1-e$ is in $\ker\varphi$. And
if we identify the subring $eA$ with $K$ we get that the previous
set coincides with
$$
\{\,\varphi\colon eA\{X\}\to K\mid\varphi\mbox{ is a
homomorphism},\:\varphi|_{eA}=Id\,\}
$$
Since $A\{X\}=A\{x_1,\ldots,x_n\}/I(X)$, then $eA\{X\}$ is of the
form
$$
K[\ldots,\sigma x_i,\ldots]/eI(X).
$$
And every homomorphism $f\colon eA\{X\}\to K$ we identify with the
point
$$
(\ldots,f(\sigma x_i),\ldots)\in X^*\subseteq
K^{n|\Sigma|}.
$$ If $\xi$ presents an element of $X$ and $\epsilon$
presents the corresponding to $\xi$ element of $X^*$ then we have
the following commutative diagram
$$
\xymatrix{
    A\{X\}\ar[r]^{\xi}\ar[d]^{\gamma_e}&A\ar[d]^{\gamma_e}\\
    K[X^*]\ar[r]^{\epsilon}&K\\
}
$$

Let $f\colon X\to Y$ be a regular mapping of pseudovarieties $X$ and
$Y$. This mapping induce a difference homomorphism $\bar f\colon
A\{Y\}\to A\{X\}$. Then applying functor $G$ we get $G(\bar f)\colon
K[Y^*]\to K[X^*]$. The last homomorphism induce a regular mapping
$f^*\colon X^*\to Y^*$.

\begin{proposition}
Let $f\colon X\to Y$ be a regular mapping of pseudovarieties $X$ and
$Y$. Then
\begin{enumerate}
\item The diagram
$$
\xymatrix{
    X\ar[r]^{f}\ar[d]^{\varphi}&Y\ar[d]^{\varphi}\\
    X^*\ar[r]^{f^*}&Y^*\\
}
$$
is commutative.
\item $\varphi$ is a homeomorphism.
\item $\varphi$ preserves unions, intersections, and complements.
\end{enumerate}
\end{proposition}
\begin{proof}
(1) The proof follows from commutativity of the diagram
$$
\xymatrix{
    A\{Y\}\ar[r]^{\bar f}\ar[d]^{\gamma_e}&A\{X\}\ar[r]^{\xi}\ar[d]^{\gamma_e}&A\ar[d]^{\gamma_e}\\
    K[Y^*]\ar[r]^{G(\bar f)}&K[X^*]\ar[r]^{\varphi(\xi)}&K\\
}
$$

(2) Let $I$ be an ideal of the ring $K[X^*]$ then $\Fun(I)$ is an
ideal of $A\{X\}$. Then it follows from Theorem~\ref{equivmod} that
difference homomorphism $\xi\colon A\{X\}\to A$ maps $F(I)$ to zero
iff the homomorphism $\varphi(\xi)\colon K[X^*]\to K$ maps the ideal
$I$ to zero. And converse, from Theorem~\ref{equivmod} we know that
every difference ideal of $A\{X\}$ has the form $\Fun(I)$ for some
ideal $I$ of $K[X^*]$.

(3) Every bijection preserves such operations.
\end{proof}

Here we recall that the set $X^*$ coincides with $X$ if we consider
an affine  pseudospace $\Fun(K)$ as an affine space $K^{|\Sigma|}$.
So, the morphism $f^*\colon X^*\to Y^*$ is just the initial mapping
$f\colon X\to Y$ if we identify $X$ with $X^*$ and $Y$ with $Y^*$.
In other words, there is no difference between pseudovarieties over
$\Fun(K)$ and algebraic varieties over $K$.

Moreover, we shall show that using this correspondence between
pseudovarieties and algebraic varieties all geometric theorems
like~\ref{imst}, \ref{constrst}, and~\ref{openst} can be derived
from the same theorems for algebraic varieties. But if we analyze
the proofs of the mentioned theorems we will see that they remain
valid even for an arbitrary pseudofield if we take pseudomaximal (or
pseudoprime) spectra instead of pseudovarieties. Therefore, we
adduce direct proofs of Theorems~\ref{imst}, \ref{constrst},
and~\ref{openst}. But now we are going to derive them from results
about algebraic varieties.

\begin{theorem}
Let $f\colon X\to Y$ be a regular mapping of pseudovarieties. Then
\begin{enumerate}
\item If $Y$ is irreducible and $f$ is dominant then the image of $f$
contains open subset.
\item For every constructible set $E\subseteq X$ the set $f(E)$ is
also constructible.
\item Let the image of $f$ be dense, then there is an open subset $U\subseteq
X$ such that the restriction of $f$ onto $U$ is an open mapping.
\end{enumerate}
\end{theorem}
\begin{proof}
(1) Consider $f^*\colon X^*\to Y^*$. Then the image of $f^*$ is
dense and $Y^*$ is irreducible. Therefore, it follows
from~\cite[chapter~5, ex.~21]{AM} that the image of $f^*$ contains
an open subset $U$. The corresponding subset $\varphi^{-1}(U)$ is an
open subset in the image of $f$.

(2) As in previous situation we consider a mapping $f^*\colon X^*\to
Y^*$. The set $E$ is of the following form $E=U_1\cap
V_1\cup\ldots\cup U_n\cap V_n$, where $U_i$ are open and $V_i$ are
closed. Since $\varphi$ is a homeomorphism we see that $\varphi(E)$
is constructible. Then it follows from~\cite[chapter~7, ex.~23]{AM}
that $f(\varphi(E))$ is constructible in $Y^*$. And applying
$\varphi^{-1}$ we conclude that $f(E)$ is constructible.

(3) Again consider the regular mapping $f^*\colon X^*\to Y^*$. We
can find an element $s\in K[Y^*]$ such that $K[Y^*]_s$ is
irreducible. Then it follows from~\cite[chapter~8, sec.~22,
th.~52]{Mu} that there is an element $u\in K[Y^*]$ such that the
ring $K[X^*]_{su}$ is faithfully flat over $K[Y^*]_{su}$. Hence,
from~\cite[chapter~7, ex.~25]{AM} we have that the mapping
$f^*\colon (X^*)_{su}\to (Y^*)_{su}$ is open. Since $\varphi$ is a
homeomorphism we get the desired result.
\end{proof}

Now suppose that $X$ is a pseudovariety with a group structure such
that all group laws are regular mappings. The last one means that
$A\{X\}$ is a Hopf algebra over $A$. So, we have
\begin{align*}
\mu^*\colon& A\{X\}\to A\{X\times X\}=A\{X\}\mathop{\otimes}_A A\{X\} \\
i^*\colon& A\{X\}\to A\{X\}\\
\varepsilon^*\colon& A\{X\}\to A\\
\end{align*}
and these mappings satisfy all necessary identities. Since functors
$\Fun$ and $G$ are equivalences they preserve limits and colimits.
Therefore, they preserve products and tensor products. So, applying
the functor $G$ to $A\{X\}$ we get a Hopf algebra $K[X^*]=G(A\{X\})$
over a field $K$, because $G$ preserves all identities on mappings
$\mu^*$, $i^*$, and $\varepsilon^*$.

\subsection{Dimension}\label{sec410}

Let $X\subseteq A^n$ be a pseudovariety over a difference closed
pseudofield $A$, and as above we suppose that $A=\Fun(K)$, where $K$
is an algebraically closed field. Since $A$ is an Artin ring and
$A\{X\}$ is a finitely generated algebra over $A$, the ring $A\{X\}$
has finite Krull dimension. Therefore, we can define $\dim X$ as
$\dim A\{X\}$.

It follows from Theorem~\ref{equivalg} that $A\{X\}=\Fun(K[X^*])$.
In other words $A\{X\}$ is a finite product of the rings $K[X^*]$.
Therefore, algebras $A\{X\}$ and $K[X^*]$ have the same Krull
dimension. So, we have the following result.

\begin{proposition}
For arbitrary pseudovariety $X$ we have $\dim X=\dim X^*$.
\end{proposition}

It should be noted that an affine pseudoline $A$ has dimension
$|\Sigma|$. Moreover, we have more general result.

\begin{proposition}
An affine pseudospace $A^n$ has dimension $n|\Sigma|$.
\end{proposition}
\begin{proof}
The ring of regular functions on $A^n$ coincides with the ring of
difference polynomials $A\{y_1,\ldots,y_n\}$. Its image under the
functor $G$ coincides with $K[\ldots,\sigma y_i,\ldots]$. Number of
variables is $n|\Sigma|$.
\end{proof}

The last result agrees with our insight about structure of $A^n$.
Indeed, every $A$ can be identified with $K^{|\Sigma|}$. So, $A^n$
coincides with $K^{n|\Sigma|}$.

\paragraph{Acknowledgement.}

The idea to write this paper appeared after a conversation with
Michael Singer in Newark in 2007. The most significant machinery
such as inheriting and the Taylor homomorphism were developed under
the influence of William Keigher's papers. Eric Rosen brought to my
attention details of ACFA theory, he explained the algebraic meaning
of logical results to me. Alexander Levin helped me learn difference
algebra. Alexey Ovchinnikov suggested to use this new geometric
theory for the Picard-Vessiot theory with difference parameters.

\bibliographystyle{plain}
\bibliography{bibs}

\end{document}